\documentclass[12pt,twoside]{amsart}

\usepackage[T1]{fontenc}
\usepackage[utf8]{inputenc}
\usepackage{lmodern}
\usepackage{amsmath,amssymb,amsthm,mathtools}
\usepackage{booktabs}
\usepackage{enumitem}
\usepackage{microtype}
\usepackage[a4paper,margin=2.6cm,top=3cm,bottom=3cm]{geometry}
\usepackage[colorlinks=true,linkcolor=black,citecolor=black,urlcolor=black]{hyperref}

\numberwithin{equation}{section}
\newtheorem{theorem}{Theorem}[section]
\newtheorem{proposition}[theorem]{Proposition}
\newtheorem{lemma}[theorem]{Lemma}
\newtheorem{corollary}[theorem]{Corollary}
\theoremstyle{definition}
\newtheorem{definition}[theorem]{Definition}

\theoremstyle{remark}
\newtheorem{remark}[theorem]{Remark}

\title[Characterization of Sobolev regularity of plurisubharmonic functions]
{Characterization of Sobolev regularity of plurisubharmonic functions}


\author{Hongrong Chen}
\address{School of Mathematics (Zhuhai), Sun Yat-Sen University, Zhuhai, Guangdong 519082, P. R. China}
 \email{chenhr73@mail2.sysu.edu.cn}

 \author{Guokuan Shao}
\address{School of Mathematics (Zhuhai), Sun Yat-Sen University, Zhuhai, Guangdong 519082, P. R. China}
\thanks{Guokuan Shao is partially supported by National Key R\& D Program of China No. 2024YFA1015200 and National Natural Science Foundation of China (Grant No. 12001549).}
 \email{shaogk@mail.sysu.edu.cn}

 \author{Wenxuan Wang}
\address{Academy of Mathematics and Systems Science, Chinese Academy of Sciences,
No. 55, Zhongguancun East Road, Haidian District, Beijing 100190, P. R. China}
\email{wxwang@amss.ac.cn}

\thanks{}

\begin{document}

\begin{abstract}
Let $f$ be a nonzero holomorphic germ at $0 \in \mathbb C^n$ with $f(0)=0$, and let $\chi$ be a $C^2$ non-decreasing convex function on the left half-line. We establish sharp necessary and sufficient conditions for the local Sobolev regularity of the plurisubharmonic function
$v=\chi(\log|f|).$ The criteria for the $L^p$-integrability of the classical Laplacian and for $W^{1,p}$-regularity are given by a weighted integral involving $\chi''$ and $\chi'$, respectively, and depend on $f$ only through the smallest multiplicity of $\operatorname{Div}(f)$. We also obtain the $W^{2,p}_{\mathrm{loc}}$ criterion for $1<p<\infty$. At the endpoint $p=1$, we prove that
\[
v\in W^{2,1}_{\mathrm{loc}}
\quad\Longleftrightarrow\quad
\chi'\in L^1((-\infty,A)),
\]
equivalently, $v$ is locally bounded. As applications, we obtain counterexamples to Calder\'on-Zygmund theory and characterize a class of functions in the local Monge--Amp\`ere domain. 
\end{abstract}

\maketitle

\tableofcontents

\section{Introduction}\label{sec:intro}

The Sobolev regularity of  plurisubharmonic functions is a natural problem in several complex variables. Even for plurisubharmonic functions with explicitly prescribed analytic singularities, the integrability of their first and second derivatives can be sensitive to the precise form of the singularity. A particularly basic class is obtained by composing logarithmic holomorphic singularities with convex functions. Namely, if $f$ is a nonzero holomorphic function and $\chi$ is convex and non-decreasing, then
$v=\chi(\log|f|)$ is plurisubharmonic after taking its trivial extension across the zero set of $f$. The regularity of $v$ is governed by two different things: the behavior of $\chi$ and the local geometry of the zero divisor of $f$. This leads naturally to the problem of determining sharp conditions on $\chi$ which characterize the Sobolev regularity of $\chi(\log|f|)$, and of understanding how these conditions depend on the multiplicities and singularities of $\operatorname{Div}(f)$.

A direct motivation for this problem comes from the endpoint of Calder\'on--Zygmund theory. For $1<p<\infty$, local $L^p$-regularity of the distributional Laplacian yields local $W^{2,p}$-regularity, whereas the corresponding implication fails at $p=1$; see, for instance, \cite{GilbargTrudinger}. Recently, Pan-Shao-Wang-Wu \cite{PSWW} showed that this endpoint failure occurs naturally for logarithmic compositions of holomorphic functions. They established energy estimates for level sets of holomorphic functions and constructed functions of the form $\log(-\log|f|)$ whose distributional Laplacians belong to $L^1_{\mathrm{loc}}$ while their full second derivatives do not. They also exhibited closely related functions $\frac{1}{\log(-\log|f|)}$  which regain the $W^{2,1}_{\mathrm{loc}}$ regularity. In complex dimension one, Chen-Shao-Wu-Xia  \cite{CSWX} obtained precise energy asymptotics and counterexample for endpoint $p=\infty$ . These results suggest that the different behaviors observed in such examples should be governed by a general condition on the outer function $\chi$, rather than by the particular logarithmic expressions used in their construction. The purpose of the present paper is to give such a characterization in a general setting. 

Let $U\subset\mathbb{C}^n$ be a neighborhood of the origin. Consider a non-constant holomorphic function $f\in\mathcal{O}(U)$ with $f(0)=0$.  Fix $A\in\mathbb{R}$ and a
non-decreasing convex function $\chi(s)\in C^2((-\infty,A])$.  After shrinking
$U$, assume that $\log|f|\le A$.  On
$U\setminus Z(f)$, where $Z(f)=\{f=0\}\neq\emptyset$, set
\begin{equation}\label{eq:def-v-intro}
v:=\chi(\log|f|).
\end{equation}
The limit of $\chi$ as $s\to -\infty$, possibly equal to $-\infty$, gives the
canonical upper-semicontinuous extension of $v$ across
$Z(f)$.  The extended function is plurisubharmonic and hence locally
integrable \cite{Demailly,Hormander}.

The germ of $f$ factors in the unique factorization ring
$\mathcal{O}_{\mathbb{C}^n,0}$ as
\begin{equation}\label{eq:factorization}
f=u\prod_{j=1}^{N}h_j^{m_j},
\end{equation}
where $u$ is a unit, the $h_j$ are distinct irreducible germs,
and $m_j\ge1$.  We write
\[
m_0=\min_{1\le j\le N}m_j
\]
for the smallest multiplicity of $\operatorname{Div}(f)$. After shrinking $U$, we choose representatives of
$u,h_1,\ldots,h_N$ such that \eqref{eq:factorization} holds throughout $U$.
Away from $Z(f)$, the classical Laplacian of $v$ is
\[
G=\chi''(\log|f|)\,|\nabla\log|f||^2.
\]
$G$ denotes its trivial extension by zero on $Z(f)$ throughout this paper.  The first main
theorem gives both the sharp integrability threshold for $G$ and characterization of
the distributional Laplacian of $v$.

\begin{theorem}
\label{thm:laplace-intro}
For every $1\le p<\infty$, one has
\begin{equation}\label{eq:laplace-criterion-intro}
G\in L^p_{\mathrm{loc}}(U)
\quad\Longleftrightarrow\quad
\int_{-\infty}^{A}|\chi''(s)|^p
   e^{\frac{2(1-p)}{m_0}s}\,\mathrm{d} s<\infty.
\end{equation}
Let
$\ell_\chi=\lim_{s\to-\infty}\chi'(s),$
which exists and is nonnegative.  If $\mu_f$ denotes the positive divisor
measure normalized by $\Delta\log|f|=2\pi\mu_f,$ 
then the distributional Laplacian on $U$ is
\begin{equation}\label{eq:full-distribution-laplace-intro}
\Delta v=G+2\pi\ell_\chi\mu_f.
\end{equation}
\end{theorem}

Theorem~\ref{thm:laplace-intro} also explains the contribution of the zero divisor to the distributional Laplacian.  The condition in \eqref{eq:laplace-criterion-intro} characterizes the local $L^p$-integrability of the classical Laplacian $G$ on $U\setminus Z(f)$, while the term $2\pi\ell_\chi\mu_f$ is the singular part supported on $\operatorname{Div}(f)$.  Consequently, the distributional Laplacian $\Delta v$ is represented by an $L^p_{\mathrm{loc}}$-function if and only if $\ell_\chi=0$ and the condition \eqref{eq:laplace-criterion-intro} holds.  When $1<p<\infty$, the local Calder\'on--Zygmund estimate then yields the $L^p$-integrability of the full  Hessian.  This argument is no longer available at $p=1$, and the $W^{2,1}_{\mathrm{loc}}$-regularity of $v$ must instead be established by estimating its second derivatives directly.

We first turn to the corresponding first-order regularity problem.  On $U\setminus Z(f)$, the chain rule gives
\[
\nabla v=\chi'(\log|f|)\nabla\log|f|,
\]
so the natural candidate for the weak gradient of $v$ is the vector field
\[
H:=\chi'(\log|f|)\nabla\log|f|,
\]
extended by zero on $Z(f)$.  For any $1 \leq p < 2$, the well-known result shows that $v \in \mathrm{W}_{loc}^{1,p}(U)$ \cite[Theorem~1.48]{GZ}.  For $p\ge2$, the next theorem gives a sharp criterion in terms of $\chi'$ and the minimal multiplicity $m_0$, and shows that the $L^p$-integrability of this classical gradient is also sufficient for first-order Sobolev regularity .

\begin{theorem}\label{thm:w1-intro}
Let $2\le p<\infty$.  The following conditions are equivalent:
\begin{equation}\label{eq:w1-criterion-intro}
v\in W^{1,p}_{\mathrm{loc}}(U)
\quad\Longleftrightarrow\quad
H\in L^p_{\mathrm{loc}}(U)
\quad\Longleftrightarrow\quad
\int_{-\infty}^{A}|\chi'(s)|^p
  e^{\frac{2-p}{m_0}s}\,\mathrm{d} s<\infty.
\end{equation}
In this case $H$ is the weak gradient of $v$.
\end{theorem}

For the Sobolev regularity $W^{2,p}$ of $v$, the residual divisor term  in Theorem \ref{thm:laplace-intro} must first vanish. Once this is imposed, the criterion is established for $1<p<\infty$, while the endpoint $p=1$ requires a different sharp criterion, as shown in the third main theorem.

\begin{theorem}\label{thm:w2-intro}
The following assertions hold.
\begin{enumerate}[label=\textup{(\roman*)}]
\item For $1<p<\infty$,
\begin{equation}\label{eq:w2p-criterion-intro}
v\in W^{2,p}_{\mathrm{loc}}(U)
\quad\Longleftrightarrow\quad
\ell_\chi=0
\quad\text{and}\quad
\int_{-\infty}^{A}|\chi''(s)|^p
 e^{\frac{2(1-p)}{m_0}s}\,\mathrm{d} s<\infty.
\end{equation}
\item At the endpoint $p=1$,
\begin{equation}\label{eq:w21-criterion-intro}
v\in W^{2,1}_{\mathrm{loc}}(U)
\quad\Longleftrightarrow\quad
\chi'\in L^1((-\infty,A))
\quad\Longleftrightarrow\quad
\lim_{s\to-\infty}\chi(s)>-\infty.
\end{equation}
\end{enumerate}
\end{theorem}

Two points are worth to note. First, although the zero set $Z(f)$ may be singular, the conditions in \eqref{eq:laplace-criterion-intro} and \eqref{eq:w1-criterion-intro} depend on $f$ only through the smallest multiplicity $m_0$ of $\operatorname{Div}(f)$.  Second, the case $p=1$ in Theorem~\ref{thm:w2-intro} is different from the case $1<p<\infty$.  For $1<p<\infty$, the condition for $v\in W^{2,p}_{\mathrm{loc}}(U)$ is expressed in terms of $\ell_\chi$ and the $L^p$-integrability of the classical Laplacian $G$.  But for $p=1$, the condition is 
\[
v\in W^{2,1}_{\mathrm{loc}}(U)
\quad\Longleftrightarrow\quad
\chi'\in L^1((-\infty,A)).
\]
In particular, the condition $G\in L^1_{\mathrm{loc}}(U)$ alone does not imply that $v\in W^{2,1}_{\mathrm{loc}}(U)$.

 We briefly describe the main ideas of the proofs.  The necessary conditions follow from the behavior of $f$ near a smooth point of an irreducible component of $\operatorname{Div}(f)$.  At such a point, after a change of holomorphic  coordinates, one may write $f=w_1^m$ locally.
The problem is then reduced to monomial case.  To prove the sufficient conditions, we use the coarea formula and estimates for the level sets of $|f|$.  The proof of the $W^{2,1}_{\mathrm{loc}}$ criterion also requires an estimate involving the second derivatives of $f$.  For this purpose, we take a log-resolution of $\operatorname{Div}(f)$ and reduce the required estimate to an integral involving monomials.  The distribution identity part are identified by integration by parts on the domains $\{|f|>\varepsilon\}$
and then letting $\varepsilon$ tend to zero.    The level-set estimates used here are related to those obtained in \cite{PSWW}.  

As an immediate consequence, the two endpoint examples considered in \cite{PSWW} are recovered from Theorem~\ref{thm:w2-intro}. The first example shows a universal phenomenon for the failure of $L^1$ regularity in Calder\'on-Zygmund theory.

\begin{corollary}[\cite{PSWW} Theorem~1.4, Theorem~1.5]\label{cor:psww}
Let $f$ be as above and assume after shrinking $U$ that $\log|f|<-e$.

\begin{enumerate}[label=\textup{(\roman*)}]
\item For
$
v_1=-\log(-\log|f|),
$
one has
\[
v_1\in W^{1,2}_{\mathrm{loc}}(U),
\qquad
\Delta v_1\in L^1_{\mathrm{loc}}(U),
\qquad
v_1\notin W^{2,1}_{\mathrm{loc}}(U).
\]

\item For
$
v_2=\frac{1}{\log(-\log|f|)},
$
one has
\[
v_2\in W^{1,2}_{\mathrm{loc}}(U),
\qquad
\Delta v_2\in L^1_{\mathrm{loc}}(U),
\qquad
v_2\in W^{2,1}_{\mathrm{loc}}(U).
\]
\end{enumerate}
\end{corollary}

Indeed, the two cases in Corollary \ref{cor:psww} correspond respectively to
\[
\chi(s)=-\log(-s)
\qquad\text{and}\qquad
\chi(s)=\frac{1}{\log(-s)}.
\]
In the first case $\chi'\notin L^1((-\infty,A))$, whereas in the second case $\chi'\in L^1((-\infty,A))$.  Thus the different endpoint behavior of the two examples is exactly characterized by the criterion in Theorem~\ref{thm:w2-intro} (ii).

Another concrete consequence is obtained from the choice $\chi(s) = e^{\gamma s}$ where $\gamma > 0$, for which $v = |f|^\gamma$. Our criteria give the sharp thresholds
\[
|f|^\gamma \in W_{\mathrm{loc}}^{1,p} \iff \gamma > \frac{p - 2}{pm_0}, \quad 2 \leq p < \infty,
\]
and
\[
|f|^\gamma \in W_{\mathrm{loc}}^{2,p} \iff \gamma > \frac{2(p - 1)}{pm_0}, \quad 1 \leq p < \infty.
\]
At the endpoint, notice that $|f|^\gamma \in W_{\mathrm{loc}}^{1,2}$ and $|f|^\gamma \in W_{\mathrm{loc}}^{2,1}$ for every $\gamma > 0$.

One fundamental problem in pluripotential theory is to define complex Monge--Amp\`ere operator for unbounded plurisubharmonic functions. Beyond locally bounded setting of Bedford-Taylor theory \cite{BT82}, different smooth decreasing approximations may produce different limiting measures. The class $\mathcal D(U)$ denotes the classical local domain, in which the Monge-Amp\`ere measure of every smooth decreasing plurisubharmonic approximation converges to the same limit. Section~\ref{sec:MA} recalls the precise definition and its relation to Cegrell's class $\mathcal E$. The following application characterizes elements in $\mathcal D(U)$ for $v$ by an integral condition on $\chi$, independent of multiplicities and singularities of $\operatorname{Div}(f)$. In dimension $2$, this condition coincides with the $W^{1,2}_{\mathrm{loc}}$ criterion in Theorem \ref{thm:w1-intro}.

\begin{theorem}\label{prop:app-D-criterion}
Assume that $\chi\in C^\infty((-\infty,A])$, the dimension $n$ is at least 2 and after
subtracting a constant that $\chi\le-2$ on $(-\infty,A]$.
Then
\begin{equation}\label{eq:app-D-criterion}
v=\chi(\log|f|)\in\mathcal D(U)
\quad\Longleftrightarrow\quad
\int_{-\infty}^{A}
(-\chi(s))^{n-2}\chi'(s)^2\,\mathrm{d}s<\infty.
\end{equation}
When these conditions hold,
\[
(dd^cv)^n=0,
\qquad
\nu(v,x)=0\quad\text{for every }x\in U,
\]
where $\nu(v,x)$ denotes the Lelong number of $v$ at $x$.
In particular, if $U$ is a sufficiently small ball, the same
integral characterizes $v\in\mathcal E(U)$.
\end{theorem}


The paper is organized as follows. In Section~\ref{sec:preliminaries}, we introduce some conventions and reduce necessity to a one-variable radial model.
Section~\ref{sec:resolution} develops the required divisorial estimates and
proves the level-set energy bounds.  The classical and distributional
Laplacians are treated in Section~\ref{sec:laplace}.  Section~\ref{sec:sobolev}
proves the first- and second-order Sobolev criteria, with a separate argument
at the endpoint $p=1$.  In Section~\ref{sec:examples}  we deal with the two applications. 

\section{Preliminaries}\label{sec:preliminaries}

\subsection{Conventions}

Throughout the paper, $U\subset\mathbb{C}^n$ is a sufficiently small neighborhood of
the origin, $f\in\mathcal{O}(U)$ is nonzero with $f(0)=0$, and
$Z(f)=\{f=0\}\neq\emptyset$.  We fix $A\in\mathbb{R}$ and a convex non-decreasing function
$\chi\in C^2(({-\infty},A])$, shrinking $U$ so that $\log|f|\le A$.  We write
$u_f=\log|f|$ and $v=\chi(u_f)$, with the canonical upper-semicontinuous
extension across $Z(f)$.  The symbols $\nabla$, $\Delta$, $D^2$, and
$\mathrm{d}$ always denote the real Euclidean differential objects under the
identification $\mathbb{C}^n\simeq\mathbb{R}^{2n}$.

Write $z_j=x_j+iy_j$ and
\[
f_j=\frac{\partial f}{\partial z_j},\qquad
f_{jk}=\frac{\partial^2f}{\partial z_j\partial z_k},\qquad
|\partial f|^2=\sum_{j=1}^n|f_j|^2.
\]
With the standard Euclidean metric, $|\nabla|f||=|\partial f|$ at points where $f\ne0$, and $|\nabla f|^2=2|\partial f|^2$ when $f$ is viewed as a
complex-valued map. 

By the convexity and monotonicity, the limit
\[
\ell_\chi=\lim_{s\to-\infty}\chi'(s)\in[0,\infty)
\]
exist.  The composition rule for plurisubharmonic functions shows that the
upper semicontinuous extension of $v$ is psh.  It is not
identically $-\infty$, because it is finite on $U\setminus Z(f)$, and hence it is
locally integrable. Moreover $v \in \mathrm{W}_{loc}^{1,p}(U)$, for any $1 \leq p < 2$ (\cite{GZ} Theorem~1.48 ).    Notice also that convexity gives a linear growth bound
\begin{equation}\label{eq:chi-linear-growth}
|\chi(s)|\le C(1+|s|),\qquad s\le A.
\end{equation}
To verify this bound, use $0\le\chi'(t)\le\chi'(A)$ for $t\le A$:
\[
0\le\chi(A)-\chi(s)
=\int_s^A\chi'(t)\,\mathrm{d} t
\le\chi'(A)(A-s).
\]
Every plurisubharmonic function which is not identically equal to $-\infty$ belongs to $L^q_{\mathrm{loc}}$ for every $1\le q<\infty$; see \cite[Theorem~1.48]{GZ}. Hence
$v=\chi(\log|f|)\in L^q_{\mathrm{loc}}(U)$ for every finite $q$.

\subsection{Differential formulas out of the zero divisor}

Put $u_f=\log|f|$.  On $U\setminus Z(f)$ the function $u_f$ is pluriharmonic,
and
\begin{equation}\label{eq:uf-derivatives}
\frac{\partial u_f}{\partial z_j}=\frac{f_j}{2f},
\qquad
\frac{\partial^2u_f}{\partial z_j\partial z_k}
=\frac{f_{jk}}{2f}-\frac{f_j f_k}{2f^2},
\qquad
\frac{\partial^2u_f}{\partial z_j\partial\overline z_k}=0.
\end{equation}
The chain rule shows
\begin{equation}\label{e92901}
\begin{split}
\nabla v
 &=\chi'(u_f)\nabla u_f,\\
D^2v
 &=\chi''(u_f)\,\,\mathrm{d} u_f\otimes\,\mathrm{d} u_f
   +\chi'(u_f)D^2u_f,\\
\Delta v
 &=\chi''(u_f)|\nabla u_f|^2.
 \end{split}
\end{equation}
Here $D^2v$ is a quatratic form and $\mathrm{d}u_f\otimes\mathrm{d}u_f$ denotes the tensor
$(X,Y)\mapsto\mathrm{d}u_f(X)\mathrm{d}u_f(Y)$.
In particular,
\begin{equation} \label{estimate of gradient}
    |\nabla v|\approx |\chi'(\log|f|)|\frac{|\partial f|}{|f|}
\end{equation}
and
\begin{equation}\label{eq:hessian-size-estimate}
|D^2v|\le C\left[
 (|\chi''(\log|f|)|+|\chi'(\log|f|)|)
     \frac{|\partial f|^2}{|f|^2}
 +|\chi'(\log|f|)|\frac{|D^2f|}{|f|}
\right].
\end{equation}
Indeed, \eqref{eq:uf-derivatives} and the equivalence of real and complex
tensor norms imply
\[
|\nabla u_f|\le C\frac{|\partial f|}{|f|},
\qquad
|D^2u_f|\le C\left(
\frac{|D^2f|}{|f|}
+\frac{|\partial f|^2}{|f|^2}
\right).
\]
Substituting into the first two identities in \eqref{e92901}, we get both estimates.

\subsection{The smooth one-variable model}\label{sec:model}

At a smooth point of the component $\{h_j=0\}$ which belongs to no other
component, such point exists for any open neighborhood of $0$, as the regular part $Z(f)_{reg}$ of is dense in $Z(f)$ see \cite{Demailly}. Then there are holomorphic coordinates $w=(w_1,\ldots,w_n)$ in which
\begin{equation}\label{eq:local-model-unit}
f(w)=a(w)w_1^{m_j},\qquad a(0)\ne0.
\end{equation}
After shrinking the coordinate neighborhood, the unit has a holomorphic
$m_j$-th root $b$ with $b^{m_j}=a$.  The map
\[
(w_1,w_2,\ldots,w_n)\longmapsto
(b(w)w_1,w_2,\ldots,w_n)
\]
has nonzero Jacobian at $p$. In the new coordinate, $f=w_1^{m_j}$.  Thus the model
\[
V_m(w)=\chi(m\log|w_1|)
\]
verifies necessary conditions appearing below. 

To justify this reduction, write $z=\Phi(w)$ for the local
coordinate change. On relatively compact neighborhoods, $D\Phi$ and $(D\Phi)^{-1}$ are bounded, and
$|\det_{\mathbb R}D\Phi|$ is bounded above and away from zero.
Thus, on $\{w_1\ne0\}$,
\[
|\nabla_z v|(\Phi(w))\approx|\nabla_w V_m(w)|,
\qquad
G(\Phi(w))
\approx
\chi''(m\log|w_1|)\frac{m^2}{|w_1|^2}.
\]
Consequently, the corresponding local $L^p$ integrability
conditions are equivalent in the two coordinate.

\begin{lemma}\label{lem:radial-model}
Let $m\ge1$ and $V_m(w)=\chi(m\log|w_1|)$.  For every $1\le p<\infty$,
the zero extension of the classical Laplacian from $\{w_1\ne0\}$ satisfies
\begin{equation}\label{eq:model-laplace}
\Delta V_m\in L^p_{\mathrm{loc}}
\quad\Longleftrightarrow\quad
\int_{-\infty}^{A}|\chi''(s)|^p
 e^{\frac{2(1-p)}{m}s}\,\mathrm{d} s<\infty.
\end{equation}
For every $1\le p<\infty$,
\begin{equation}\label{eq:model-gradient}
\nabla V_m\in L^p_{\mathrm{loc}}
\quad\Longleftrightarrow\quad
\int_{-\infty}^{A}|\chi'(s)|^p
 e^{\frac{2-p}{m}s}\,\mathrm{d} s<\infty.
\end{equation}
Finally,
\begin{equation}\label{eq:model-w21-necessary}
V_m\in W^{2,1}_{\mathrm{loc}}\quad\Longrightarrow\quad
\chi'\in L^1(({-\infty},A)).
\end{equation}
\end{lemma}

\begin{proof}
Write $w_1=re^{i\theta}$ and set $F(r)=\chi(m\log r)$. Direct differentiation shows
\[
F'(r)=\frac{m}{r}\chi'(m\log r),
\qquad
F''(r)=\frac{m^2\chi''(m\log r)-m\chi'(m\log r)}{r^2}.
\]
Since $V_m$ depends only on $w_1$, its Laplacian reduces to the
Laplacian in the $w_1$-plane. Write $w_1=x+iy$ and
$r=(x^2+y^2)^{1/2}$. On $\{w_1\ne0\}$, we have
\begin{align*}
\Delta V_m
&=\frac{\partial^2}{\partial x^2}F(r)
  +\frac{\partial^2}{\partial y^2}F(r)\\
&=F''(r)\frac{x^2+y^2}{r^2}
  +F'(r)\left(\frac{2}{r}-\frac{x^2+y^2}{r^3}\right)\\
&=F''(r)+\frac{F'(r)}{r}\\
&=\frac{m^2}{r^2}\chi''(m\log r).
\end{align*}
The remaining variables $w_2,\ldots,w_n$ contribute only a fixed
positive volume factor when integrating over a product polydisc.
Consequently, on a sufficiently small product polydisc
$P=D(0,r_0)\times P'$, polar integration implies
\begin{align*}
\|\Delta V_m\|_{L^p(P)}^p
&=
2\pi m^{2p}\operatorname{Vol}(P')
\int_0^{r_0}
|\chi''(m\log r)|^p r^{1-2p}\,\mathrm{d}r\\
&=
2\pi m^{2p-1}\operatorname{Vol}(P')
\int_{-\infty}^{m\log r_0}
|\chi''(s)|^p
e^{\frac{2(1-p)}{m}s}\,\mathrm{d}s,
\end{align*}
where the second equality follows from the change of variables
$s=m\log r$. Since changing the finite upper endpoint does not affect
convergence, \eqref{eq:model-laplace} follows.

Similarly, $|\nabla V_m|=|F'(r)|$, and polar integration on the
same product polydisc $P=D(0,r_0)\times P'$ gives
\begin{align*}
\|\nabla V_m\|_{L^p(P)}^p
&=2\pi m^p\operatorname{Vol}(P')
  \int_0^{r_0}
  |\chi'(m\log r)|^p r^{1-p}\,\mathrm{d}r\\
&=2\pi m^{p-1}\operatorname{Vol}(P')
  \int_{-\infty}^{m\log r_0}
  |\chi'(s)|^p e^{\frac{2-p}{m}s}\,\mathrm{d}s.
\end{align*}
Then \eqref{eq:model-gradient} follows.

For $w_1=re^{i\theta}$, let $e_\theta=ie^{i\theta}$ be the unit vector
tangent to the circle $\{|w_1|=r\}$. Along the straight line
$w_1+t e_\theta$, we have
\[
|w_1+t e_\theta|=\sqrt{r^2+t^2},
\]
and therefore
\[
D^2V_m(e_\theta,e_\theta)
=
\left.\frac{\mathrm{d}^2}{\mathrm{d}t^2}\right|_{t=0}
F\bigl(\sqrt{r^2+t^2}\bigr)
=
\frac{F'(r)}{r}.
\]
Since the norm of the Hessian dominates every second directional
derivative and $\chi'\ge0$, it follows that
\[
|D^2V_m|
\ge \frac{F'(r)}{r}
=\frac{m}{r^2}\chi'(m\log r).
\]
If $V_m\in W^{2,1}_{\mathrm{loc}}$, then its classical Hessian away from
$\{w_1=0\}$ is integrable.  Polar integration therefore yields
\[
\infty>
\int_0^{r_0}|D^2V_m|r\,\mathrm{d} r
\ge
m\int_0^{r_0}\chi'(m\log r)\frac{\mathrm{d} r}{r}
=\int_{-\infty}^{m\log r_0}\chi'(s)\,\mathrm{d} s,
\]
which is exactly \eqref{eq:model-w21-necessary}.
\end{proof}

\subsection{Coarea formula and level-set notation}

For a non-negative measurable function $\Phi$, we have the following coarea formula
\cite{EvansGariepy,Federer}
\begin{equation}\label{eq:coarea}
\int_{K\cap\{0<|f|<r_0\}}\Phi\,\mathrm{d}V
=
\int_0^{r_0}
\left(
\int_{K\cap\{|f|=r\}}
\frac{\Phi}{|\partial f|}\,\mathrm{d}S
\right)\mathrm{d}r,
\end{equation}
where $K\Subset U$ and $r_0>0$ is sufficiently small.
By \cite[Lemma~2.3]{PSWW}, these level sets are smooth
real-analytic hypersurfaces near $K$ for every $0<r\le r_0$,
so the level-set integrals are understood in the usual sense.

In the local integrability estimates below, it suffices to
integrate over $K\cap\{0<|f|<r_0\}$: the integrands in this paper are bounded on $K\cap\{|f|\ge r_0\}$, whose
contribution is therefore finite. The number $r_0$ may depend
on $K$ and will be decreased when necessary.  We also use the Dirac delta notation for the same formula
\begin{equation}\label{eq:delta-coarea}
\int_{K\cap\{|f|=r\}}
\frac{\Psi}{|\partial f|}\,\mathrm{d}S
=
\int_K\Psi\,\delta(|f|-r)\,\mathrm{d}V.
\end{equation}
To justify this identity in the distributional sense with respect to $r$, let $\Psi$ be smooth and take
$\eta\in C_0^\infty((0,r_0))$. By the definition of the Dirac delta, we have
\begin{align*}
&\int_0^{r_0}\eta(r)
\left(
\int_K\Psi(z)\delta(|f(z)|-r)\,\mathrm{d}V(z)
\right)\mathrm{d}r\\
&\qquad=
\int_K\Psi(z)\eta(|f(z)|)\,\mathrm{d}V(z)\\
&\qquad=
\int_0^{r_0}\eta(r)
\left(
\int_{K\cap\{|f|=r\}}
\frac{\Psi}{|\partial f|}\,\mathrm{d}S
\right)\mathrm{d}r,
\end{align*}
where the last equality follows from \eqref{eq:coarea}.
Thus \eqref{eq:delta-coarea} is the distributional formulation
of the coarea formula.

For distributional integration by parts, fix a relatively
compact smooth domain $\Omega\Subset U$ and set
\begin{equation}\label{eq:truncated-domain}
\Omega_\varepsilon=\Omega\cap\{|f|>\varepsilon\}.
\end{equation}
By Sard's theorem, choose $\varepsilon_j\rightarrow0$ such that
each $\varepsilon_j$ is a regular value of both $|f|$ on
$\Omega\setminus Z(f)$ and its restriction to
$\partial\Omega\setminus Z(f)$.
The moving boundary then meets $\partial\Omega$ transversely,
so $\Omega_{\varepsilon_j}$ has piecewise smooth Lipschitz
boundary. All subsequent integrations by parts on
$\Omega_\varepsilon$ and the associated limits are understood
along this sequence, with $\varepsilon=\varepsilon_j$.
Since the test functions are supported away from
$\partial\Omega$, only the boundary
$\Omega\cap\{|f|=\varepsilon\}$ contributes.

\section{Divisorial estimates and small-level-set energies}
\label{sec:resolution}

The Hessian estimate \eqref{eq:hessian-size-estimate}, together with the
coarea formula, shows that the endpoint $W^{2,1}$ argument requires, in
addition to the first-order energy estimates, uniform control of
\[
\int_{K\cap\{|f|=p\}}
\frac{|D^2f|}{|\partial f|}\,\mathrm{d}S
\]
as $p\to0$. Near a smooth point of $Z(f)$, one can choose holomorphic
coordinates in which $f=w_1^m$, as explained in
Section~\ref{sec:model}, and estimate the integral directly.
To obtain a uniform bound near the singular points of $Z(f)$,
we pass to a log-resolution of $\operatorname{Div}(f)$ and
compare the vanishing orders of $f$, its derivatives, and
the Jacobian of the resolution.

\subsection{Divisorial orders of derivatives}

After shrinking $U$, choose a log-resolution
\[
\pi:Y\longrightarrow U
\]
of $\operatorname{Div}(f)$. Thus $Y$ is smooth, $\pi$ is proper
and bimeromorphic, its restriction over $U\setminus Z(f)$ is
biholomorphic, and $\operatorname{Div}(f\circ\pi)$ has simple
normal crossings; see \cite{Hironaka,Kollar}.

Let
\[
J_f=(f_1,\ldots,f_n)\subset\mathcal O_U
\]
be the Jacobian ideal of $f$, where
$f_j=\partial f/\partial z_j$.
For a prime divisor $E\subset Y$ and a nonzero holomorphic
function $g$ on $U$, set$\nu_E(g):=\operatorname{ord}_E(g\circ\pi).$
We write
\[
\alpha_E=\nu_E(f),\qquad
\beta_E=\nu_E(J_f):=\min_{1\le j\le n}\nu_E(f_j),\qquad
A_U(E)=1+\nu_E(K_Y-\pi^*K_U).
\]
Let $\kappa=\nu_E(K_Y-\pi^*K_U)$ and choose
\[
x\in E_{\mathrm{reg}}
\setminus\operatorname{Supp}(K_Y-\pi^*K_U-\kappa E).
\]
We write $\operatorname{jac}(\pi)=\det D\pi$.
There are local coordinates $y=(y_1,\ldots,y_n)$ centered
at $x$ such that
\[
E=\{y_1=0\},\qquad
\operatorname{jac}(\pi)=y_1^\kappa j_0(y),\qquad j_0(x)\ne0.
\]
In these coordinates, $A_U(E)=\kappa+1$.
If $D_j$ is the strict transform of $\{h_j=0\}$, then
\[
\alpha_{D_j}=m_j,\qquad
\beta_{D_j}=m_j-1,\qquad
A_U(D_j)=1.
\]

Proposition~\ref{prop:derivative-discrepancy} shows that
differentiation by a holomorphic vector field decreases
the vanishing order along $E$ by at most $A_U(E)$.
Corollary~\ref{cor:first-second-valuations} applies this
bound to the first and second derivatives of $f$.
Proposition~\ref{prop:exceptional-strictness} shows that
these inequalities are strict when $E$ is exceptional.
Note that the local order computations and the rank argument along
exceptional divisors will be used in the proof of
Lemma~\ref{lem:hessian-exponents}.

\begin{proposition}\label{prop:derivative-discrepancy}
Let $X$ be a complex manifold, let $E$ be a prime divisor on a smooth
birational model $\pi:Y\to X$, and let $\xi$ be a local holomorphic vector
field on $X$.  For every nonzero holomorphic function $g$, we have
\begin{equation}\label{eq:derivative-discrepancy}
\nu_E(\xi (g))+A_X(E)\ge \nu_E(g).
\end{equation}
\end{proposition}

\begin{proof}
Choose $x\in E_{\mathrm{reg}}$ outside the intersections of $E$ with all other irreducible components of $\operatorname{Div}(g\circ\pi)$ and $\operatorname{Div}(\operatorname{jac}(\pi))$.  Choose holomorphic coordinates $y=(y_1,\ldots,y_n)$ centered at $x$
such that $E=\{y_1=0\}$, and write
\[
g\circ\pi=y_1^\alpha u(y),\qquad u(x)\ne0,
\]
where $\alpha=\nu_E(g)$.  Take coordinate
$z=(z_1,\ldots,z_n)$ on $X$ near the image of this point $\pi(x)$ and write
\[
\xi=\sum_{i=1}^n a_i(z)\frac{\partial}{\partial z_i},
\qquad
D\pi=\left(\frac{\partial z_i}{\partial y_p}\right).
\]
If $\kappa=\nu_E(\operatorname{jac}(\pi))$, then $A_X(E)=\kappa+1$.  On the locus where $\pi$ is
biholomorphic, the meromorphic pullback of $\xi$ is
\[
\widetilde\xi=\pi^*\xi
=\sum_{p=1}^n B_p(y)\frac{\partial}{\partial y_p},
\qquad
(B_1,\ldots,B_n)^{\mathsf T}
=(D\pi)^{-1}(a_1\circ\pi,\ldots,a_n\circ\pi)^{\mathsf T}.
\]
Since $(D\pi)^{-1}=\frac{\operatorname{adj}(D\pi)}{\operatorname{jac}(\pi)}$, where $\operatorname{adj}(D\pi)$ is the adjugate matrix of $D\pi$, 
and every elements of $\operatorname{adj}(D\pi)$ is holomorphic, each coefficient
satisfies $\nu_E(B_p)\ge-\kappa$.  On the other hand,
\[
\frac{\partial(g\circ\pi)}{\partial y_1}
=y_1^{\alpha-1}
\left(\alpha u+y_1\frac{\partial u}{\partial y_1}\right),
\]
so its $E$-order is at least $\alpha-1$, while for $p\ge2$,
\[
\frac{\partial(g\circ\pi)}{\partial y_p}
=y_1^\alpha\frac{\partial u}{\partial y_p}
\]
has $E$-order at least $\alpha$.  It follows term by term that
\[
\nu_E\left(
B_1\frac{\partial(g\circ\pi)}{\partial y_1}
\right)\ge\alpha-\kappa-1,
\qquad
\nu_E\left(
B_p\frac{\partial(g\circ\pi)}{\partial y_p}
\right)\ge\alpha-\kappa
\quad (p\ge2).
\]
On the open set where $\pi$ is biholomorphic, it follows from the definition of
$\widetilde\xi=\pi^*\xi$ and the chain rule that
\[
\widetilde\xi(g\circ\pi)=(\xi (g))\circ\pi.
\]
The left-hand side is meromorphic on the whole chart, whereas the
right-hand side is holomorphic. Since they agree on a dense open
subset, they represent the same meromorphic function on the chart.
Consequently, their orders along $E$ coincide. We
obtain
\[
\nu_E(\xi (g))
=\nu_E\bigl(\widetilde\xi(g\circ\pi)\bigr)
\ge\alpha-\kappa-1
=\nu_E(g)-A_X(E).
\]
This completes the proof.
\end{proof}

\begin{corollary}
\label{cor:first-second-valuations}
For every prime divisor $E$ over $U$,
\begin{equation}\label{eq:first-second-valuations}
\beta_E+A_U(E)\ge\alpha_E,
\qquad
\nu_E(f_{ij})+2A_U(E)\ge\alpha_E.
\end{equation}
\end{corollary}

\begin{proof}
For each $i$, applying Proposition~\ref{prop:derivative-discrepancy} to
$g=f$ and $\xi=\partial/\partial z_i$ implies
\[
\nu_E(f_i)+A_U(E)\ge\nu_E(f)=\alpha_E.
\]
Taking the minimum over $i$ proves
$\beta_E+A_U(E)\ge\alpha_E$.  Applying the same proposition once more, now
to $g=f_i$ and $\xi=\partial/\partial z_j$, yields
\[
\nu_E(f_{ij})+A_U(E)\ge\nu_E(f_i)
\ge\alpha_E-A_U(E).
\]
Thus $\nu_E(f_{ij})+2A_U(E)\ge\alpha_E$.  If one of the derivatives
vanishes identically, its vanishing order is $+\infty$; hence the same
inequality still holds under the usual convention
$+\infty+c=+\infty\ge\alpha_E$.
\end{proof}

For exceptional divisors, the inequalities in Corollary~\ref{cor:first-second-valuations} are strict, as the following proposition shows.

\begin{proposition}\label{prop:exceptional-strictness}
If $E$ is exceptional divisor of resolution $\pi$, then
\[
\nu_E(f_i)+A_U(E)>\alpha_E,
\qquad
\nu_E(f_{ij})+2A_U(E)>\alpha_E
\]
for all $i,j$.
\end{proposition}

\begin{proof}
Choose a smooth point $x\in E$ lying on no other component of either
$\operatorname{Div}(f\circ\pi)$ or the Jacobian divisor, and at which
$d\pi|_{TE}$ has its generic rank. Choose local coordinates
$y=(y_1,\ldots,y_n)$ centered at $x$ such that $E=\{y_1=0\}$. We may write
\[
\operatorname{jac}(\pi)=y_1^\kappa j_0(y),\qquad j_0(x)\ne0.
\]
Then $\kappa=\operatorname{ord}_E(K_Y-\pi^*K_U)$, and hence
$A_U(E)=\kappa+1$.  We first prove that the pullback of every 
holomorphic vector field $\xi$ on $U$ has the form
\begin{equation}\label{eq:exceptional-vector-field}
\pi^*\xi
=y_1^{-\kappa}\left(
a_1(y)y_1\frac{\partial}{\partial y_1}
+\sum_{p=2}^n a_p(y)\frac{\partial}{\partial y_p}
\right),
\end{equation}
where the $a_p$ are holomorphic for all $p \in \{1,2,\dots,n\}$.

It is enough to check this for a coordinate vector field
$\partial/\partial z_i$.  On the biholomorphic locus,
\[
\pi^*\frac{\partial}{\partial z_i}
=\sum_{p=1}^n ((D\pi)^{-1})_{p i}
\frac{\partial}{\partial y_p},
\qquad
((D\pi)^{-1})_{p i}
=\frac{\operatorname{adj}(D\pi)_{p i}}{y_1^\kappa j_0(y)}.
\]
The coefficient of $\partial/\partial y_1$ involves the cofactor
$\operatorname{adj}(D\pi)_{1i}$.  This cofactor is, up to sign, an
$(n-1)\times(n-1)$ minor of the matrix obtained by deleting the column
corresponding to the normal direction $\partial/\partial y_1$.  Since $E=\{y_1=0\}$, its tangent space is spanned by
$\partial/\partial y_2,\ldots,\partial/\partial y_n$.
Thus deleting the first column of $D\pi$ gives the matrix of
\[
d\pi_x|_{T_xE}:T_xE\longrightarrow T_{\pi(x)}U.
\]
Because $E$ is exceptional, $\dim_{\mathbb C}\pi(E)\le n-2$,
so this matrix has rank at most $n-2$ at a generic point of $E$.
Its $(n-1)\times(n-1)$ minors therefore vanish on a dense open
subset of $E$, and hence on all of $E$ because it is holomorphic. Consequently
\[
\operatorname{adj}(D\pi)_{1i}=y_1 b_i(y)
\]
for a holomorphic function $b_i$.  No additional factor is needed for the
other cofactors.  Dividing by $y_1^\kappa j_0(y)$ proves
\eqref{eq:exceptional-vector-field}.

Now write $f\circ\pi=y_1^\alpha u(y),$
where $u$ is a holomorphic unit near $x$.  Applying
\eqref{eq:exceptional-vector-field} gives
\begin{align*}
(\xi f)\circ\pi
&=y_1^{-\kappa}\left[
a_1y_1\frac{\partial}{\partial y_1}(y_1^\alpha u)
+\sum_{p=2}^n
a_p\frac{\partial}{\partial y_p}(y_1^\alpha u)
\right]\\
&=y_1^{\alpha-\kappa}\left[
a_1\left(\alpha u+y_1\frac{\partial u}{\partial y_1}\right)
+\sum_{p=2}^n a_p\frac{\partial u}{\partial y_p}
\right].
\end{align*}
Therefore
\[
\nu_E(\xi f)\ge\alpha-\kappa
=\alpha-A_U(E)+1.
\]
Taking $\xi=\partial/\partial z_i$ proves
\[
\nu_E(f_i)+A_U(E)\ge\alpha_E+1>\alpha_E.
\]

The preceding computation applies to any nonzero holomorphic function.
Applying it to $f_i$ and $\xi=\partial/\partial z_j$, and then using
the first-order estimate, we obtain
\[
\nu_E(f_{ij})
\ge \nu_E(f_i)-A_U(E)+1
\ge \alpha_E-2A_U(E)+2.
\]
Thus
\[
\nu_E(f_{ij})+2A_U(E)\ge\alpha_E+2>\alpha_E.
\]
If $f_i$ vanishes identically, the conclusion follows under the
convention $\nu_E(0)=+\infty$.
\end{proof}

\subsection{A monomial delta integral}

 This subsection is devoted to an elementary calculus result that will be used in the proof of Proposition \ref{prop:second-energy}.

\begin{lemma}\label{lem:monomial-delta}
Let $a_i>0$, $b_i>-2$, $0<R<1$, and
\[
I(t)=\int_{(0,R)^k}\prod_{i=1}^k r_i^{b_i+1}
        \delta(\prod_{i=1}^k r_i^{a_i}-t)\,\mathrm{d} r_1\cdots\,\mathrm{d} r_k.
\]
Put $\lambda_i=(b_i+2)/a_i$.
If $\lambda_i\ge1$ for all $i$ and $\#\{i \in \{1,...,k\}:\lambda_i =1\} \leq 1$, then
$I(t)\le C$ for $0<t<t_0$.
\end{lemma}

\begin{proof}
Put
\[
x_i=-\log r_i,\qquad T=-\log t,\qquad L=-\log R.
\]
Then $x_i>L$, $r_i=e^{-x_i}$, and
$r_i^{b_i+1}\,\mathrm{d} r_i=e^{-(b_i+2)x_i}\,\mathrm{d} x_i$ after reversing
the limits.  Moreover,
\[
\delta\left(e^{-\sum_i a_i x_i}-e^{-T}\right)
=e^T\delta\left(\sum_i a_i x_i-T\right),
\]
because the derivative of $s\mapsto e^{-s}$ at $s=T$ has absolute value
$e^{-T}$.  Under the change of variables $y_i=a_ix_i$, we have
\[
\mathrm{d}x_1\cdots\mathrm{d}x_k
=\frac{1}{a_1\cdots a_k}\,
\mathrm{d}y_1\cdots\mathrm{d}y_k,
\qquad
\sum_{i=1}^k(b_i+2)x_i
=\sum_{i=1}^k\lambda_i y_i.
\]
Moreover, $x_i>L$ becomes $y_i>a_iL$. Therefore,
\begin{align*}
I(t)
&=e^T\int_{(L,\infty)^k}
e^{-\sum_i(b_i+2)x_i}
\delta\left(\sum_i a_ix_i-T\right)
\,\mathrm{d}x_1\cdots\mathrm{d}x_k\\
&=\frac{e^T}{a_1\cdots a_k}
\int_{\{y_i>a_iL\}}
e^{-\sum_i\lambda_i y_i}
\delta\left(\sum_i y_i-T\right)
\,\mathrm{d}y_1\cdots\mathrm{d}y_k.
\end{align*}
Since
\[
\nabla\left(\sum_{i=1}^k y_i\right)=(1,\ldots,1),
\qquad
\left|\nabla\left(\sum_{i=1}^k y_i\right)\right|=\sqrt{k},
\]
by the coarea formula, we have
\[
I(t)
=\frac{e^T}{\sqrt{k}\,a_1\cdots a_k}
\int_{\substack{\sum_i y_i=T\\y_i>a_iL}}
e^{-\sum_i\lambda_i y_i}\,
\mathrm{d}S.
\]
  Since $\sum_i y_i=T$, this becomes
\begin{equation}\label{eq:monomial-simplex}
I(t)=\frac{1}{\sqrt{k}\,a_1\cdots a_k}
\int_{\substack{\sum_i y_i=T\\y_i>a_iL}}
e^{-\sum_i(\lambda_i-1)y_i}\,
\mathrm{d}S.
\end{equation}

Assume first that $\lambda_i\ge1$ for all $i$ and that equality holds for
exactly one index.  Relabel so that $\lambda_1=1$.
The orthogonal projection of the hyperplane $\sum_i y_i=T$ onto the
$(y_2,\ldots,y_k)$-coordinate changes surface measure by the fixed factor
$\sqrt{k}$.  
Then we have
\begin{align*}
I(t)
&\le C
\int_{\substack{y_i>a_iL,\ i\ge2\\
\sum_{i=2}^k y_i<T-a_1L}}
e^{-\sum_{i=2}^k(\lambda_i-1)y_i}\,
\mathrm{d} y_2\cdots\mathrm{d} y_k\\
&\le C\prod_{i=2}^k
\int_{a_iL}^{\infty}e^{-(\lambda_i-1)y_i}\,\mathrm{d} y_i
<\infty.
\end{align*}
The bound is independent of $T$.

If all $\lambda_i>1$, let
$\eta=\min_i(\lambda_i-1)>0$.  The integrand in
\eqref{eq:monomial-simplex} is at most $e^{-\eta T}$.  The $(k-1)$-dimensional measure of the integration region is bounded by
$C(1+T)^{k-1}$.  Thus
\[
I(t)\le C e^{-\eta T}(1+T)^{k-1},
\]
which is uniformly bounded for sufficiently large $T$. The proof is completed.
\end{proof}

If $\lambda_i=1$ for exactly $q\ge2$ indices and $\lambda_i>1$ for
the remaining indices, then the integral in
\eqref{eq:monomial-simplex} grows on the order of
$T^{q-1}=|\log t|^{q-1}$ as $t\to0$. Hence the uniform bound in
Lemma~\ref{lem:monomial-delta} can fail when equality holds for more
than one index.

\subsection{Second-order level-set energy estimates}

This subsection supplies the uniform second-order level-set bound used in the
endpoint proof of Theorem~\ref{thm:w2-intro}.  On a sufficiently small resolution chart, choose local holomorphic
coordinates $y=(y_1,\ldots,y_n)$ in which $f\circ\pi$ has monomial
form. We write
\begin{equation} \label{eq:snc-chart}
    F:=f\circ\pi=y_1^{a_1}\cdots y_k^{a_k},
    \qquad
    \operatorname{jac}(\pi)
    =j_0(y)y_1^{\kappa_1}\cdots y_k^{\kappa_k},
\end{equation}
where $j_0$ is a unit. We use $i,j$ for the $z$-coordinate indices,
$p,q$ for the $y$-coordinate indices, and $\ell$ for the divisor
components, with $1\le i,j,p,q\le n$ and $1\le\ell\le k$.
Here we calculate the expression of $\operatorname{jac}(\pi)^2(f_{ij}\circ\pi)$. Fix $1\leq i,j\leq n$. Since
\[
    (D\pi)^{-1}
    =\frac{\operatorname{adj}(D\pi)}{\operatorname{jac}(\pi)}
\]
on the biholomorphic locus, the lifted coordinate vector field is
\[
    \pi^*\left(\frac{\partial}{\partial z_i}\right)
    =
    \sum_{p=1}^n
    ((D\pi)^{-1})_{p i}\frac{\partial}{\partial y_p}
    =
    \sum_{p=1}^n
    \frac{\operatorname{adj}(D\pi)_{p i}}{\operatorname{jac}(\pi)}\frac{\partial}{\partial y_p}.
\]
Consequently,
\[
    (f_j)\circ\pi
    =
    \sum_{q=1}^n
    \frac{\operatorname{adj}(D\pi)_{q j}}{\operatorname{jac}(\pi)}F_q.
\]
Applying the lifted vector field corresponding to
$\partial/\partial z_i$ and using the product rule, we obtain
\begin{align*}
    (f_{ij})\circ\pi
    &=
    \sum_{p=1}^n
    \frac{\operatorname{adj}(D\pi)_{p i}}{\operatorname{jac}(\pi)}
    \partial_p
    \left(
        \sum_{q=1}^n
        \frac{\operatorname{adj}(D\pi)_{q j}}{\operatorname{jac}(\pi)}F_q
    \right)\\
    &=
    \sum_{p,q=1}^n
    \frac{\operatorname{adj}(D\pi)_{p i}\operatorname{adj}(D\pi)_{q j}}{\operatorname{jac}(\pi)^2}F_{pq}
    +
    \sum_{p,q=1}^n
    \frac{\operatorname{adj}(D\pi)_{p i}}{\operatorname{jac}(\pi)}
    \partial_p\left(\frac{\operatorname{adj}(D\pi)_{q j}}{\operatorname{jac}(\pi)}\right)F_q.
\end{align*}
Since
\[
    \partial_p\left(\frac{\operatorname{adj}(D\pi)_{q j}}{\operatorname{jac}(\pi)}\right)
    =
    \frac{\partial_p \operatorname{adj}(D\pi)_{q j}}{\operatorname{jac}(\pi)}
    -
    \frac{\operatorname{adj}(D\pi)_{q j}\,\partial_p \operatorname{jac}(\pi)}{\operatorname{jac}(\pi)^2},
\]
it follows that
\begin{align*}
    (f_{ij})\circ\pi
    &=
    \frac{1}{\operatorname{jac}(\pi)^2}
    \sum_{p,q=1}^n
    \operatorname{adj}(D\pi)_{p i}\operatorname{adj}(D\pi)_{q j}F_{pq}\\
    &\quad+
    \frac{1}{\operatorname{jac}(\pi)^2}
    \sum_{p,q=1}^n
    \operatorname{adj}(D\pi)_{p i}(\partial_p \operatorname{adj}(D\pi)_{q j})F_q\\
    &\quad-
    \frac{1}{\operatorname{jac}(\pi)^3}
    \sum_{p,q=1}^n
    \operatorname{adj}(D\pi)_{p i}\operatorname{adj}(D\pi)_{q j}(\partial_p \operatorname{jac}(\pi))F_q.
\end{align*}
Multiplying by $\operatorname{jac}(\pi)^2$, we have
\[
    \operatorname{jac}(\pi)^2(f_{ij}\circ\pi)
    =
    \sum_{p,q=1}^n
    \left(
        T_{pq}^{(1)}
        +T_{pq}^{(2)}
        +T_{pq}^{(3)}
    \right),
\]
where
\[
\begin{aligned}
    T_{pq}^{(1)}
        &:=\operatorname{adj}(D\pi)_{p i}\operatorname{adj}(D\pi)_{q j}F_{pq},\\
    T_{pq}^{(2)}
        &:=\operatorname{adj}(D\pi)_{p i}(\partial_p \operatorname{adj}(D\pi)_{q j})F_q,\\
    T_{pq}^{(3)}
        &:=-\operatorname{adj}(D\pi)_{p i}\operatorname{adj}(D\pi)_{q j}
          \frac{\partial_p \operatorname{jac}(\pi)}{\operatorname{jac}(\pi)}F_q.
\end{aligned}
\]
The first two terms are holomorphic outside of exceptional divisor.  The third term is initially defined only on the
biholomorphic locus, but its apparent poles along the exceptional
divisor are removable.  Indeed, since
$
    \operatorname{jac}(\pi)=j_0(y)\prod_{\ell=1}^k y_\ell^{\kappa_\ell},
$
we have
$
    \frac{\partial_p \operatorname{jac}(\pi)}{\operatorname{jac}(\pi)}
    =
    \frac{\partial_p j_0}{j_0}
    +
    \frac{\kappa_p}{y_p}.
$
If $\kappa_p=0$, this expression is holomorphic.  If
$\kappa_p>0$, then $E_p=\{y_p=0\}$ is exceptional, and the rank
argument from Proposition~\ref{prop:exceptional-strictness} gives
$
    y_p\mid \operatorname{adj}(D\pi)_{p i}.
$
Consequently,
\[
    \operatorname{adj}(D\pi)_{p i}
    \frac{\partial_p\operatorname{jac}(\pi)}{\operatorname{jac}(\pi)}
    =
    \frac{\operatorname{adj}(D\pi)_{p i}}{y_p}
    \left(
        y_p\frac{\partial_p j_0}{j_0}
        +\kappa_p
    \right),
\]
which is holomorphic because
$\operatorname{adj}(D\pi)_{p i}/y_p$ is holomorphic.  Thus every $T_{pq}^{(\alpha)}$ extends
holomorphically to the whole chart.  We continue to denote these
extensions by $T_{pq}^{(\alpha)}$.

The following lemma gives the required termwise
control.
\begin{lemma}\label{lem:hessian-exponents}
Fix a relatively compact chart as above.
There exists a constant $C>0$, for every $\alpha\in\{1,2,3\}$ and every nonzero term
$T_{pq}^{(\alpha)}$, there exist integers
$b_{pq,\ell}^{(\alpha)}>-2$, $1\le\ell\le k$, such that
\[
\bigl|T_{pq}^{(\alpha)}(y)\bigr|
\le
C\prod_{\ell=1}^k
|y_\ell|^{b_{pq,\ell}^{(\alpha)}}.
\]
Moreover, $b_{pq,\ell}^{(\alpha)}+2\ge a_\ell$
for every $\ell$, and equality can occur for at most one index $\ell$
for each fixed term $T_{pq}^{(\alpha)}$. The constant $C$ is independent of $y$ and can be chosen uniformly for all $1\le i,j,p,q\le n$ and $\alpha\in\{1,2,3\}$.
\end{lemma}

\begin{proof}
Since $F=\prod_{\ell=1}^k y_\ell^{a_\ell},$ for $1\leq q \leq k$ we have
\[
|F_q|
\le C\prod_{\ell=1}^k
|y_\ell|^{a_\ell-\delta_{q\ell}},
\qquad
|F_{pq}|
\le C\prod_{\ell=1}^k
|y_\ell|^{a_\ell-\delta_{p\ell}-\delta_{q\ell}}.
\]
All the remaining factors in $T_{pq}^{(1)}$ and
$T_{pq}^{(2)}$ are holomorphic. The computation preceding this
lemma shows that $ \operatorname{adj}(D\pi)_{p i}\frac{\partial_p \operatorname{jac}(\pi)}{\operatorname{jac}(\pi)} $
is also holomorphic, so the remaining factors in
$T_{pq}^{(3)}$ are holomorphic as well. On the fixed relatively compact chart, these holomorphic
factors are bounded. Since there are only finitely many
choices of $i,j,p,q$ and $\alpha$, their bounds can be
absorbed into a single constant $C$, depending only on
$f$, $\pi$, and the chosen coordinate chart.

For $\alpha=1$, set
\[
b_{pq,\ell}^{(1)}
=
a_\ell-\delta_{p\ell}-\delta_{q\ell},
\]
while for $\alpha=2,3$, set
\[
b_{pq,\ell}^{(\alpha)}
=
a_\ell-\delta_{q\ell}.
\]
The required pointwise bounds are deduced from these choices. Since $a_\ell\ge1$,
all the exponents are greater than $-2$. Moreover,
\[
b_{pq,\ell}^{(\alpha)}+2\ge a_\ell.
\]
For $\alpha=2,3$ the inequality is always strict. For $\alpha=1$, equality
requires $p=q=\ell$, which can occur for at most one index $\ell$.
\end{proof}

Lemma 3.5 provides the precise termwise divisorial control needed to estimate the second derivatives on a resolution chart. Combining this local estimate with the monomial delta integral of Lemma 3.4, we now obtain the uniform second-order level-set bound.

\begin{proposition}\label{prop:second-energy}
For every $K\Subset U$, there exist $C,t_0>0$ independent of $t$, such that, for every
$0<t<t_0$,
\begin{equation}\label{eq:second-energy}
\int_{K\cap\{|f|=t\}}\frac{|D^2f|}{|\partial f|}\,\mathrm{d}S\le C.
\end{equation}
\end{proposition}

\begin{proof}
Since $|D^2f|\le C\sum_{i,j=1}^n|f_{ij}|$, it suffices to estimate
\[
\int_{K\cap\{|f|=t\}}
\frac{|f_{ij}|}{|\partial f|}\,\mathrm{d}S
=\int_K|f_{ij}|\,\delta(|f|-t)\,\mathrm{d}V
\]
for each fixed pair $i,j$.

We first work on a single relatively compact polydisc
$P$ on the resolution, normalized as in \eqref{eq:snc-chart}, with $F=f\circ\pi=\prod_{\ell=1}^k y_\ell^{a_\ell}$.
All the exponents and constants in the following local computation
refer to this fixed chart. The expansion preceding
Lemma~\ref{lem:hessian-exponents} shows
\[
|f_{ij}\circ\pi|\,|\operatorname{jac}(\pi)|^2
=\left|\operatorname{jac}(\pi)^2(f_{ij}\circ\pi)\right|
\le\sum_{\alpha=1}^3\sum_{p,q=1}^n|T_{pq}^{(\alpha)}|.
\]
For each nonzero term, Lemma~\ref{lem:hessian-exponents} yields
\[
|T_{pq}^{(\alpha)}(y)|
\le C\prod_{\ell=1}^k
|y_\ell|^{b_{pq,\ell}^{(\alpha)}}.
\]
Writing $y_\ell=\rho_\ell e^{i\theta_\ell}$ and integrating in
the angular variables and in $y_{k+1},\ldots,y_n$, we obtain
\begin{align*}
&\int_P|T_{pq}^{(\alpha)}|
\delta(|F|-t)\,\mathrm{d}V_y\\
&\qquad\le C\int_{(0,R)^k}
\prod_{\ell=1}^k\rho_\ell^{b_{pq,\ell}^{(\alpha)}+1}
\delta\!\left(\prod_{\ell=1}^k\rho_\ell^{a_\ell}-t\right)
\,\mathrm{d}\rho_1\cdots\mathrm{d}\rho_k,
\end{align*}
where $0<R<1$ is fixed by the chosen polydisc.
For this term, the same lemma gives
\[
\frac{b_{pq,\ell}^{(\alpha)}+2}{a_\ell}\ge1,
\qquad
\#\left\{\ell:
\frac{b_{pq,\ell}^{(\alpha)}+2}{a_\ell}=1\right\}\le1.
\]
Lemma~\ref{lem:monomial-delta} therefore bounds this integral
uniformly for all sufficiently small $t>0$. Summing over the
finitely many nonzero terms, we conclude that
\[
\int_P|f_{ij}\circ\pi|\,|\operatorname{jac}(\pi)|^2
\delta(|F|-t)\,\mathrm{d}V_y\le C_P,
\]
where $C_P$ depends on $f$, $\pi$, and the chosen coordinate
polydisc, but is independent of $t$. Since there are finitely many
pairs $i,j$, the bound and the range of $t$ may be chosen uniformly
in these indices.

We now pass to $K$. By properness of $\pi$, the set
$\pi^{-1}(K)\cap\{f\circ\pi=0\}$ is compact and is covered by
finitely many polydiscs.
The complement of their union in $\pi^{-1}(K)$ is compact and
disjoint from $\{f\circ\pi=0\}$. Hence, for all sufficiently
small $t>0$, the set
\[
\pi^{-1}(K)\cap\{|f\circ\pi|=t\}
\]
is contained in that union. Since the integrands are nonnegative,
by the change-of-variables formula and the sum of the finitely many
local bounds, we have
\[
\int_K|f_{ij}|\,\delta(|f|-t)\,\mathrm{d}V
\le C,
\qquad 0<t<t_0,
\]
for constants $C,t_0>0$ independent of $t$ and of $i,j$.
Summing over $i,j$ proves \eqref{eq:second-energy}.
\end{proof}

\begin{remark}
Proposition  \ref{prop:second-energy} will be used specifically in the proof of the endpoint $W^{2,1}$-criterion in Theorem \ref{thm:w2-intro} (ii). Indeed, the Hessian estimate \eqref{eq:hessian-size-estimate} contains the term$|\chi'(\log |f|)|\frac{|D^2f|}{|f|},$which is not controlled by the first-order level-set estimates alone. Proposition \ref{prop:second-energy}, together with the coarea formula, reduces its $L^1$-integrability to the condition $\int_{-\infty}^{A}\chi'(s) \mathrm{d}s<\infty.$Thus the second-order level-set estimate is precisely what allows us to treat the full Hessian directly at the endpoint $p=1$.
\end{remark}

\subsection{First-order level-set energy estimates}

We next establish first-order energy estimates on small level and
sublevel sets of $|f|$. These estimates provide the bounds needed in
the coarea arguments and control the integrals over the moving boundary
$\{|f|=\varepsilon\}$ when integration by parts on
$\{|f|>\varepsilon\}$ is followed by the limit $\varepsilon\to0$.
They are derived from the continuity of the currents of integration
over the fibers $\{f=t\}$.

\begin{proposition}\label{prop:first-energy}
Let $B\Subset U$ be a concentric Euclidean ball chosen as above. Then
\[
\begin{aligned}
 \int_{B\cap\{|f|=r\}}|\partial f|\,dS
   =2\pi\mu_0(B)\,r+o(r) ,
 \quad \quad
 \int_{B\cap\{|f|<r\}}|\partial f|^2\,dV
   =\pi\mu_0(B)\,r^2+o(r^2),
\end{aligned}
\]
as $r\rightarrow 0$. 

Consequently, for every compact subset $K\Subset U$, there exist
constants $C>0$ and $r_0>0$ such that, for $0<r<r_0$,
\[
\begin{aligned}
 \int_{K\cap\{|f|=r\}}|\partial f|\,dS &\le Cr ,
 \quad \quad
 \int_{K\cap\{|f|<r\}}|\partial f|^2\,dV &\le Cr^2.
\end{aligned}
\]
\end{proposition}

\begin{proof}
After shrinking $U$, we may assume that $U$ is a Euclidean ball
centered at the origin and that $f$ is holomorphic on a neighborhood of
$\overline U$.  Fix $K\Subset U$, and choose a concentric Euclidean ball $B$
such that $K\Subset B\Subset U$.
For $t\in\mathbb C$, set
\[
F_t:=\{z\in U:f(z)=t\},
\qquad
\mu_t:=[F_t]\wedge
\frac{\omega_{\mathrm{Euc}}^{n-1}}{(n-1)!},
\]
where $\omega_{\mathrm{Euc}}$ is the standard Euclidean K\"ahler form.
Thus $\mu_t$ is the multiplicity-weighted Euclidean volume measure on
the fiber $F_t$.

We first observe that
$
t\longmapsto \mu_t(B)
$
is continuous for $t$ sufficiently close to $0$. Indeed, if
$t_j\to t$, then
\[
\log|f-t_j|\longrightarrow \log|f-t|
\qquad\text{in }L^1_{\mathrm{loc}}(U).
\]
By the Poincar\'e--Lelong formula,
$
[F_{t_j}]\longrightarrow [F_t]
$
in the sense of currents, and hence
$
\mu_{t_j}\rightharpoonup\mu_t
$
weakly as locally finite measures.

For a Euclidean ball $B$, one has $\mu_t(\partial B)=0$.  Indeed, write
$[F_t]=\sum_jm_{j,t}[Z_{j,t}]$ in terms of reduced irreducible components and
multiplicities.  If $n\ge2$, let $\rho(z)=|z|^2-R^2$, so that
$\partial B=\{\rho=0\}$. The restriction of $\rho$ to the regular
part of each positive-dimensional component $Z_{j,t}$ is strictly
plurisubharmonic and therefore cannot vanish identically on any
nonempty open subset. Since this restriction is real analytic, its
zero set has zero $(2n-2)$-dimensional measure. The singular part of
$Z_{j,t}$ also has zero $(2n-2)$-dimensional measure, and hence $\mu_t(\partial B)=0.$  If $n=1$, choose the radius so that $F_0\cap\partial B=\varnothing$; the same remains true for small $t$.  Thus $\mu_t(\partial B)=0$ in either case, including for nonreduced fibers.
Therefore, by the Portmanteau theorem,
\[
\mu_t(B)
\le
\liminf_{j\to\infty}\mu_{t_j}(B)
\le
\limsup_{j\to\infty}\mu_{t_j}(B)
\le
\mu_t(\overline B)
=
\mu_t(B).
\]
Consequently,
$
\mu_{t_j}(B)\longrightarrow\mu_t(B).
$

We now establish uniform two-sided bounds for the volumes of
$F_t\cap B$ for all sufficiently small nonzero values of $t$. Since $f(0)=0$ and $f$ is not identically zero, the divisor
$\operatorname{Div}(f)$ is nonempty in $B$. Writing
\[
\operatorname{Div}(f)
=
\sum_{j=1}^{N} m_j Z_j
\]
locally in $B$, where the $Z_j$ are the irreducible components and
$m_j\geq 1$, we have
\[
\mu_0(B)
=
\sum_{j=1}^{N}
m_j
\int_{Z_j\cap B}
\frac{\omega_{\mathrm{Euc}}^{n-1}}{(n-1)!}.
\]
At least one of the components $Z_j$ passes through the origin.
Since every nonempty analytic hypersurface has positive
$(2n-2)$-dimensional volume on a sufficiently small neighborhood of
each of its points, while analytic hypersurfaces have locally finite
volume, it follows that $0<\mu_0(B)<\infty$.

By the continuity of the fiber mass at $t=0$ proved above,
\[
\mu_t(B)\longrightarrow \mu_0(B)=\mu_0(B)
\qquad\text{as }t\to0.
\]
Hence, taking for instance $\eta=\mu_0(B)/2$ in the definition of
continuity, there exists $r_1>0$ such that
\[
|\mu_t(B)-\mu_0(B)|<\frac{\mu_0(B)}{2}
\qquad\text{whenever }|t|<r_1.
\]
Consequently,
\begin{equation}
\frac{\mu_0(B)}{2}
\leq
\mu_t(B)
\leq
\frac{3\mu_0(B)}{2},
\qquad |t|<r_1.
\label{eq:fiber-mass-two-sided}
\end{equation}

On the other hand, after decreasing $r_1$ if necessary, the regularity
of sufficiently small nonzero level sets
\cite[Lemma~2.3]{PSWW} ensures that
\[
\partial f(z)\neq0
\qquad
\text{whenever }z\in\overline B
\quad\text{and}\quad
0<|f(z)|<r_1.
\]
Thus every $t$ with $0<|t|<r_1$ is a regular value of $f$ on
$\overline B$, and
\[
F_t\cap B=\{f=t\}\cap B
\]
is a smooth complex hypersurface. In particular, the divisor of
$f-t$ is reduced along $F_t\cap B$, so the current of integration
$[F_t]$ carries multiplicity one there. Therefore
\[
\mu_t(B)
=
\int_{F_t\cap B}
\frac{\omega_{\mathrm{Euc}}^{n-1}}{(n-1)!}
=
\operatorname{Vol}_{2n-2}(F_t\cap B).
\]
Combining this identity with \eqref{eq:fiber-mass-two-sided}, we obtain
\begin{equation}
0<
\frac{\mu_0(B)}{2}
\leq
\operatorname{Vol}_{2n-2}(F_t\cap B)
\leq
\frac{3\mu_0(B)}{2}
<\infty,
\qquad
0<|t|<r_1.
\label{eq:fiber-volume-two-sided}
\end{equation}
Thus the volumes of all sufficiently small nonzero fibers are uniformly bounded away from zero.

We now apply the coarea formula from the image side. Define
\[
I_B(r)
:=
\int_{B\cap\{|f|=r\}}
|\partial f|\,dS,
\qquad
0<r<r_0.
\]
Since $B\cap\{|f|=r\}$ is a smooth real hypersurface and
\[
f:
B\cap\{|f|=r\}
\longrightarrow
\{w\in\mathbb C:|w|=r\}
\]
is a submersion, the coarea formula gives
\[
I_B(r)
=
\int_{|w|=r}
\operatorname{Vol}_{2n-2}(F_w\cap B)
\,dS.
\]
Equivalently,
\[
I_B(r)
=
\int_{|w|=r}\mu_w(B)\,dS.
\]
Writing $w=re^{i\theta}$, so that
$dS(w)=r\,d\theta$, we obtain
\[
I_B(r)
=
r\int_0^{2\pi}
\mu_{re^{i\theta}}(B)\,d\theta.
\]
Since $t\mapsto\mu_t(B)$ is continuous at the origin,
\[
\sup_{\theta\in[0,2\pi]}
\left|
\mu_{re^{i\theta}}(B)-\mu_0(B)
\right|
\longrightarrow0
\qquad
\text{as }r\to0^+.
\]
It follows that
\[
\frac{I_B(r)}{r}
\longrightarrow
2\pi\mu_0(B),
\]
or equivalently,
\[
I_B(r)
=
2\pi\mu_0(B)\,r+o(r)
\approx r.
\]

For completeness, let
\[
J_B(r)
:=
\int_{B\cap\{|f|<r\}}
|\partial f|^2\,dV.
\]
Applying the scalar coarea formula \eqref{eq:coarea}, together with
$|\nabla|f||=|\partial f|$, yields
\[
J_B(r)
=
\int_0^r I_B(s)\,ds.
\]
Consequently,
\[
J_B(r)
=
\pi\mu_0(B)\,r^2+o(r^2)
\approx r^2.
\]

Now let $K\Subset U$. Choosing the ball $B$ above so that
$K\Subset B\Subset U$, the positivity of the integrands gives
\[
\begin{aligned}
 \int_{K\cap\{|f|=r\}}|\partial f|\,dS
 \le I_B(r), 
 \quad \quad
 \int_{K\cap\{|f|<r\}}|\partial f|^2\,dV
 \le J_B(r).
\end{aligned}
\]
Thus the asymptotics on $B$ yield uniform upper bounds on every
compact subset $K\Subset U$. This completes the proof.
\end{proof}

For comparison, \cite{PSWW} establishes the corresponding upper bounds of orders $O(r)$ and $O(r^2)$. Proposition 3.8 shows in addition that, for the particular nondegenerate ball $B$ chosen above, these bounds have nonzero leading terms and therefore give the exact orders of the two energy quantities. This stronger conclusion is specific to $B$; for an arbitrary compact subset $K\Subset U$, the corresponding lower bounds may fail, and only the upper bounds are asserted.

In addition to the energy estimates above, we need the following
unweighted bounds for the Hausdorff measure of small level sets and the volume of the corresponding sublevel sets.

\begin{proposition}[\cite{PSWW}~Theorem 1.2]\label{area estimate}
Under the same hypotheses as in Proposition~\ref{prop:first-energy}, there
exist constants $r_0>0$, $\gamma\in(0,1]$, and $\tau\in(0,2]$ such that
\[
\mathcal{H}^{2n-1}\bigl(\{z\in U:|f|=r\}\bigr)=O(r^\gamma),
\qquad
\operatorname{Vol}\bigl(\{z\in U:|f|<r\}\bigr)=O(r^\tau)
\]
for $0<r<r_0$. Here $\mathcal{H}^{2n-1}$ means $2n-1$ dimensional Hausdorff measure. 
\end{proposition}

\begin{remark}\label{rem:weighted-versus-unweighted}
Proposition~\ref{prop:first-energy} controls  level and sublevel set energies, with linear and quadratic main terms for the chosen ball
$B$.  On the other hand, Proposition~\ref{area estimate} controls the
Hausdorff area and volume, with decay exponents
$\gamma$ and $\tau$.  The level and sublevel set energy estimates are used to
identify weak derivatives. 
\end{remark}

\section{The Laplacian and its residual divisor measure}\label{sec:laplace}

We now prove Theorem~\ref{thm:laplace-intro}.  Necessity is proved by the
monumial case at a smooth point of a component of multiplicity $m_0$.
Sufficiency combines a pointwise gradient bound with the level-set energy estimate. We then pass from the classical
Laplacian on $U\setminus Z(f)$ to the full distributional Laplacian
on $U$.

\begin{lemma}\label{lem:gradient-f}
After shrinking $U$, there is a constant $C$ such that
\begin{equation}\label{eq:gradient-f}
|\partial f(z)|\le C|f(z)|^{1-1/m_0},
\qquad z\in U.
\end{equation}
\end{lemma}

\begin{proof}
On $U\setminus Z(f)$, logarithmic differentiation of
\eqref{eq:factorization} gives
\begin{equation}\label{eq:log-differentiate-f}
\frac{\partial_i f}{f}
=\frac{\partial_i u}{u}+\sum_{j=1}^{N}m_j\frac{\partial_i h_j}{h_j}.
\end{equation}
Shrink $U$ so that $u, u^{-1}$, the first derivatives of $u$ and $h_j$, and
$|h_j|$ are bounded, with $|h_j|\le1$.  For every fixed $j$,
\[
|f|^{1/m_0}
=|u|^{1/m_0}\prod_k|h_k|^{m_k/m_0}
\le C|h_j|^{m_j/m_0}\le C|h_j|.
\]
Thus $|h_j|^{-1}\le C|f|^{-1/m_0}$.  Equation
\eqref{eq:log-differentiate-f} now yields
\[
\frac{|\partial_i f|}{|f|}
\le C(1+|f|^{-1/m_0})
\le C|f|^{-1/m_0},
\]
after shrinking $U$ once more. This proves \eqref{eq:gradient-f} on $U\setminus Z(f)$ . If $m_0=1$, \eqref{eq:gradient-f} is obviously holds for entire $U$, if $m_0 > 1$, both side of \eqref{eq:gradient-f} are 0.  The proof is completed. 
\end{proof}

\begin{proof}[Proof the $L^p$ criterion of Theorem~\ref{thm:laplace-intro}]
Suppose first that $G\in L^p_{\mathrm{loc}}(U)$.  Choose an irreducible component
$\{h_j=0\}$ with $m_j=m_0$ and a smooth point belonging to no other
component.  By the coordinate reduction \eqref{eq:local-model-unit}, we may
assume locally that $f=w_1^{m_0}$.  Lemma~\ref{lem:radial-model} then gives
\[
\int_{-\infty}^{A}|\chi''(s)|^p
 e^{\frac{2(1-p)}{m_0}s}\,\mathrm{d} s<\infty.
\]

Conversely, fix $K\Subset U$. By \eqref{eq:uf-derivatives}, \eqref{e92901} and the coarea formula,
\begin{align*}
\int_{K\cap\{0<|f|<r_0\}}|G|^p\,\mathrm{d}V
&\le C\int_{K\cap\{0<|f|<r_0\}}
|\chi''(\log|f|)|^p
\frac{|\partial f|^{2p}}{|f|^{2p}}\,\mathrm{d}V\\
&=C\int_0^{r_0}
|\chi''(\log r)|^p r^{-2p}
\left(
\int_{K\cap\{|f|=r\}}
|\partial f|^{2p-1}\,\mathrm{d}S
\right)\mathrm{d}r.
\end{align*}
Because $2p-2\ge0$, we deduce fro Lemma~\ref{lem:gradient-f} and
Proposition \ref{prop:first-energy} that
\begin{align*}
\int_{K\cap\{0<|f|<r_0\}}|G|^p\,\mathrm{d}V
&\le C\int_0^{r_0}|\chi''(\log r)|^p r^{-2p} r^{(2p-2)(1-1/m_0)}
 \left(\int_{K\cap\{|f|=r\}}|\partial f| \,\mathrm{d}S\right)\,\mathrm{d} r\\
&\le C\int_0^{r_0}|\chi''(\log r)|^p
 r^{-2p+(2p-2)(1-1/m_0)+1}\,\mathrm{d} r\\
&=C\int_0^{r_0}|\chi''(\log r)|^p
 r^{-1-2(p-1)/m_0}\,\mathrm{d} r.
\end{align*}
The change of variables $s=\log r$ turns the last expression into
\[
C\int_{-\infty}^{\log r_0}|\chi''(s)|^p
 e^{\frac{2(1-p)}{m_0}s}\,\mathrm{d} s,
\]
which is finite by hypothesis.  Since $K$ was arbitrary, $G\in L^p_{\mathrm{loc}}(U)$.
\end{proof}

\subsection{The distributional Laplacian}

We first remove the divisor under the condition $\ell_\chi=0$ .

\begin{proposition}
\label{prop:no-residual-measure}
If $\ell_\chi=0$, then
\begin{equation}\label{eq:no-residual-measure}
\Delta v=G
\end{equation}
in the sense of distribution.
\end{proposition}

\begin{proof}
Let $\varphi\in C_0^\infty(U)$ and choose a smooth domain
$\Omega\Subset U$ such that
$\operatorname{supp}\varphi\Subset\Omega$.
For $\varepsilon=\varepsilon_j$ chosen as above, Green's
formula applies on $\Omega_\varepsilon$.
Since $\varphi$ vanishes near $\partial\Omega$, we obtain
\begin{equation}\label{eq:green-truncated}
\begin{aligned}
\int_{\Omega_\varepsilon}v\Delta\varphi\,\mathrm{d} V
={}&\int_{\Omega_\varepsilon}G\varphi\,\mathrm{d} V\\
&+\int_{\Omega\cap\{|f|=\varepsilon\}}
   (v\partial_\nu\varphi-\varphi\partial_\nu v)\,\mathrm{d}S .
\end{aligned}
\end{equation}
By \eqref{eq:chi-linear-growth} and Proposition  \ref{area estimate},
\[
\left|\int_{\{|f|=\varepsilon\}}v\partial_\nu\varphi\,\mathrm{d}S \right|
\le  C(1+|\log\varepsilon|)\varepsilon^\gamma\to0 
\]
Moreover,
\[
|\partial_\nu v|
\le |\chi'(\log\varepsilon)|\frac{|\partial f|}{\varepsilon}.
\]
Therefore Proposition  \ref{prop:first-energy} yields
\begin{align*}
    \left|\int_{\{|f|=\varepsilon\}}\varphi\partial_\nu v\,\mathrm{d}S \right|
&\le C\frac{|\chi'(\log\varepsilon)|}{\varepsilon}\int_{\{|f|=\varepsilon\}}|\partial f|\,\mathrm{d}S  \\
&\le C\chi'(\log\varepsilon)\longrightarrow0,
\end{align*}
because $\ell_\chi=0$.  Let
$\varepsilon \rightarrow0$.  Since $v\in L^1_{\mathrm{loc}}(U)$,
\[
\int_{\Omega_{\varepsilon}}v\Delta\varphi\,\mathrm{d} V
\longrightarrow\int_Uv\Delta\varphi\,\mathrm{d} V.
\]
The $p=1$ case of
\eqref{eq:laplace-criterion-intro} gives $G\in L^1_{\mathrm{loc}}(U)$, and
therefore
\[
\int_{\Omega_{\varepsilon}}G\varphi\,\mathrm{d} V
\longrightarrow\int_UG\varphi\,\mathrm{d} V.
\]
Both boundary integrals tend to zero by the estimates above.  Passing to the
limit in \eqref{eq:green-truncated} yields
\[
\int_Uv\Delta\varphi\,\mathrm{d} V
=\int_UG\varphi\,\mathrm{d} V
\]
for every test function $\varphi$, which is precisely
\eqref{eq:no-residual-measure}.
\end{proof}

\begin{proof}[Completion of the proof of Theorem~\ref{thm:laplace-intro}]
Define
\[
\widetilde\chi(s)=\chi(s)-\ell_\chi s.
\]
Since $\chi'$ is non-decreasing and tends to $\ell_\chi$ at $-\infty$,
$\widetilde\chi$ is again convex and non-decreasing,
$\widetilde\chi''=\chi''$, and
$\lim_{s\to-\infty}\widetilde\chi'(s)=0$.  Proposition
\ref{prop:no-residual-measure} applied to
$\widetilde v=\widetilde\chi(\log|f|)$ yields
\[
\Delta\widetilde v=G.
\]
Since $v=\widetilde v+\ell_\chi\log|f|$, the Poincar\'e--Lelong formula, traced with
the Euclidean K\"ahler form, implies
\[
\Delta v=G+\ell_\chi\Delta\log|f|=G+2\pi\ell_\chi\mu_f.
\]
The divisor measure $\mu_f$ is nonzero because $f(0)=0$.  Thus the concentrated
term vanishes exactly when $\ell_\chi=0$.
\end{proof}

\begin{remark}\label{rem:divisor-mass}
Formula \eqref{eq:full-distribution-laplace-intro} explains the role of the
boundary condition.  Integrability of the classical function $G$ alone says
nothing about a measure already concentrated on the zero divisor.  The number
$\ell_\chi$ is exactly the coefficient of that measure.
\end{remark}

\section{Sobolev regularity}\label{sec:sobolev}
In this section, we prove Theorem \ref{thm:w1-intro} and Theorem \ref{thm:w2-intro}.
\subsection{First-order regularity for exponents at least two}

\begin{proof}[Proof of Theorem~\ref{thm:w1-intro}]
Suppose first that $v\in W^{1,p}_{\mathrm{loc}}(U)$.  On the open set
$U\setminus Z(f)$, the function is of $C^2$ class, so its weak gradient
coincides there with the classical vector field $H$.  Since
$Z(f)$ has Lebesgue measure zero, the zero extension of $H$ belongs
to $L^p_{\mathrm{loc}}(U)$.  Next, suppose that $H\in L^p_{\rm loc}(U)$.  Choose an irreducible component
$\{h_j=0\}$ with $m_j=m_0$ and a smooth point which lies on no other
component.  By \eqref{eq:local-model-unit}, after taking a holomorphic
$m_0$-th root of the unit and changing the normal coordinate, one has
$f=w_1^{m_0}$.  Thus \eqref{eq:model-gradient} gives
\[
\int_{-\infty}^{A}|\chi'(s)|^p
e^{(2-p)s/m_0}\,\mathrm{d} s<\infty.
\]

Now assume this integral is finite and fix $K\Subset U$.  By coarea formula we get
\begin{align*}
\int_{K\cap\{0<|f|<r_0\}}|H|^p\,\mathrm{d}V
&\le C\int_0^{r_0}|\chi'(\log r)|^p r^{-p}
 \left(\int_{K\cap\{|f|=r\}}|\partial f|^{p-1}\,\mathrm{d}S \right)\,\mathrm{d} r.
\end{align*}
Because $p\ge2$, we may leave one copy of $|\partial f|$ in the level
integral and estimate the remaining $p-2$ copies by
Lemma~\ref{lem:gradient-f}.  Using Proposition  \ref{prop:first-energy}, we obtain
\begin{align*}
\int_{K\cap\{0<|f|<r_0\}}|H|^p\,\mathrm{d}V
&\le C\int_0^{r_0}|\chi'(\log r)|^p r^{-p}r^{(p-2)(1-1/m_0)}
 \left(\int_{K\cap\{|f|=r\}}|\partial f|\,\mathrm{d}S \right)\,\mathrm{d} r\\
&\le C\int_0^{r_0}|\chi'(\log r)|^p
 r^{-p+(p-2)(1-1/m_0)+1}\,\mathrm{d} r\\
&=C\int_0^{r_0}|\chi'(\log r)|^p
 r^{-1-(p-2)/m_0}\,\mathrm{d} r\\
&=C\int_{-\infty}^{\log r_0}|\chi'(s)|^p
 e^{\frac{2-p}{m_0}s}\,\mathrm{d} s<\infty.
\end{align*}

It remains to identify $H$ as the weak gradient.
Let $\varphi\in C_0^\infty(U)$ and choose a smooth domain
$\Omega\Subset U$ with
$\operatorname{supp}\varphi\Subset\Omega$.
For $\varepsilon=\varepsilon_j$ chosen as above,
integration by parts on $\Omega_\varepsilon$ gives
\[
\int_{\Omega_\varepsilon}v\,\partial_j\varphi\,\mathrm{d} V
=-\int_{\Omega_\varepsilon}H_j\varphi\,\mathrm{d} V
+\int_{\Omega\cap\{|f|=\varepsilon\}}
v\varphi\nu_j\,\mathrm{d}S .
\]
There is no contribution from the fixed part $\partial \Omega$ of the boundary because
$\varphi$ vanishes there.  The remain boundary term is bounded by
\[
C\|\varphi\|_{L^\infty}|\chi(\log\varepsilon)|
 \mathcal{H}^{2n-1}(\Omega\cap\{|f|=\varepsilon\}),
\]
which tends to zero by \eqref{eq:chi-linear-growth} and
Proposition  \ref{area estimate}.  Since $v\in L^1_{\mathrm{loc}}$ and
$H\in L^p_{\mathrm{loc}}\subset L^1_{\mathrm{loc}}$, the two volume
integrals converge as $\varepsilon \rightarrow 0$.  Therefore
\[
\int_U v\,\partial_j\varphi\,\mathrm{d} V
=-\int_U H_j\varphi\,\mathrm{d} V.
\]
Thus $H=\nabla v$ distributionally and $v\in W^{1,p}_{\mathrm{loc}}(U)$.
\end{proof}

\subsection{Second-order regularity for $L^p_{loc},1<p<\infty$}

\begin{proposition}\label{prop:w2p}
Let $1<p<\infty$.  Then
\[
v\in W^{2,p}_{\mathrm{loc}}(U)
\quad\Longleftrightarrow\quad
\ell_\chi=0
\quad\text{and}\quad
G\in L^p_{\mathrm{loc}}(U).
\]
\end{proposition}

\begin{proof}
Suppose first that $\ell_\chi=0$ and
$G\in L^p_{\mathrm{loc}}(U)$.  Then
by Theorem~\ref{thm:laplace-intro} we have
\[
\Delta v=G
\qquad\text{in }\mathcal D'(U).
\]
Since $v\in L^p_{\mathrm{loc}}(U)$ by
\eqref{eq:chi-linear-growth}, the local Calderón-Zygmund
regularity theorem for the Poisson equation implies that
$v\in W^{2,p}_{\mathrm{loc}}(U)$; see
\cite[Theorem~9.9]{GilbargTrudinger}.

Conversely, suppose that
$v\in W^{2,p}_{\mathrm{loc}}(U)$, and write
$\Delta v=F$ with $F\in L^p_{\mathrm{loc}}(U)$.  By
\eqref{eq:full-distribution-laplace-intro},
\[
F\,\mathrm dV
=
G\,\mathrm dV+2\pi\ell_\chi\mu_f
\]
as locally finite measures.  Here $G\in L^1_{\mathrm{loc}}(U)$
by the $p=1$ part of Theorem~\ref{thm:laplace-intro}, so both
$F\,\mathrm dV$ and $G\,\mathrm dV$ are absolutely continuous
with respect to $\mathrm dV$, whereas $\mu_f$ is singular with
respect to $\mathrm dV$, because $\mu_f$ is supported on the
Lebesgue-null set $Z(f)$.  The uniqueness of the Lebesgue
decomposition \cite[Theorem~6.10]{RudinRealComplex} therefore
gives
\[
F\,\mathrm dV=G\,\mathrm dV,
\qquad
\ell_\chi\mu_f=0.
\]
Since $f(0)=0$, the divisor measure $\mu_f$ is nonzero, and hence
$\ell_\chi=0$.  Moreover, $G=F$ almost everywhere, so
$G\in L^p_{\mathrm{loc}}(U)$.
\end{proof}

Combining Proposition~\ref{prop:w2p} with
Theorem~\ref{thm:laplace-intro} proves part~(i) of
Theorem~\ref{thm:w2-intro}.

\subsection{Second-order regularity for $L^1_{loc}$}

The endpoint argument estimates the full Hessian rather than only its trace.

\begin{proof}[Proof of Theorem~\ref{thm:w2-intro}(ii)]
Assume first that $v\in W^{2,1}_{\mathrm{loc}}(U)$.  Choose a smooth point of
a component of multiplicity $m_0$ which lies on no other component.  As in
\eqref{eq:local-model-unit}, a  change of holomorphic coordinate
reduces $f$ to $w_1^{m_0}$.  Put $r=|w_1|$ and
$F(r)=\chi(m_0\log r)$.  Using Lemma~\ref{lem:radial-model} we obtain
$\chi'\in L^1((-\infty,A))$.

Conversely, suppose $\chi'\in L^1(({-\infty},A))$.  Since $\chi'$ is
nonnegative and non-decreasing,
\begin{equation}\label{eq:ell-zero-from-L1}
\ell_\chi=0,
\qquad
\int_{-\infty}^{A}\chi''(s)\,\mathrm{d} s=\chi'(A)<\infty.
\end{equation}
Fix $K\Subset U$. To estimate the right-hand side of
\eqref{eq:hessian-size-estimate} near $Z(f)$, set
\begin{align*}
I_1&=
\int_{K\cap\{0<|f|<r_0\}}
|\chi''(\log|f|)|
\frac{|\partial f|^2}{|f|^2}\,\mathrm{d}V,\\
I_2&=
\int_{K\cap\{0<|f|<r_0\}}
|\chi'(\log|f|)|
\frac{|\partial f|^2}{|f|^2}\,\mathrm{d}V,\\
I_3&=
\int_{K\cap\{0<|f|<r_0\}}
|\chi'(\log|f|)|
\frac{|D^2f|}{|f|}\,\mathrm{d}V.
\end{align*}
From the $p=1$ case of Theorem~\ref{thm:laplace-intro} and
\eqref{eq:ell-zero-from-L1} we obtain $I_1<\infty$.  For the second term, coarea
and Proposition  \ref{prop:first-energy} imply
\begin{align*}
I_2
&\le C\int_0^{r_0}\chi'(\log r)r^{-2}
 \left(\int_{K\cap\{|f|=r\}}|\partial f|\,\mathrm{d}S \right)\,\mathrm{d} r
\le C\int_{-\infty}^{\log r_0}\chi'(s)\,\mathrm{d} s<\infty.
\end{align*}
For the third term, Proposition~\ref{prop:second-energy} yields
\begin{align*}
I_3
&\le\int_0^{r_0}\frac{\chi'(\log r)}r
 \left(\int_{K\cap\{|f|=r\}}
       \frac{|D^2f|}{|\partial f|}\,\mathrm{d}S \right)\,\mathrm{d} r
\le C\int_{-\infty}^{\log r_0}\chi'(s)\,\mathrm{d} s<\infty.
\end{align*}
Thus every classical second derivative of $v$ on $U\setminus Z(f)$, extended
by zero on $Z(f)$, belongs to $L^1_{\mathrm{loc}}(U)$.

It remains to identify these $L^1$ functions with the weak second
derivatives. Since $0\le\chi'(s)\le\chi'(A)$ and
$\chi'\in L^1((-\infty,A))$, we have
\[
\int_{-\infty}^{A}|\chi'(s)|^2\,\mathrm{d}s
\le\chi'(A)\int_{-\infty}^{A}\chi'(s)\,\mathrm{d}s<\infty.
\]
Theorem~\ref{thm:w1-intro} with $p=2$ therefore shows
$v\in W^{1,2}_{\mathrm{loc}}(U)$, with weak gradient $H$.
Fix $\varphi\in C_0^\infty(U)$ and a smooth domain
$\Omega\Subset U$ containing its support. Let $L_{ij}$ denote the zero extension of the classical derivative
$\partial_iH_j$ from $U\setminus Z(f)$.  The estimates for
$I_1,I_2,I_3$ show that $L_{ij}\in L^1_{\mathrm{loc}}(U)$. For $\varepsilon=\varepsilon_j$ chosen as above, a second integration by parts yields
\[
\int_{\Omega_\varepsilon}L_{ij}\varphi\,\mathrm{d} V
=-\int_{\Omega_\varepsilon}H_j\partial_i\varphi\,\mathrm{d} V
+\int_{\Omega\cap\{|f|=\varepsilon\}}
H_j\varphi\nu_i\,\mathrm{d}S.
\]
By \ref{estimate of gradient} and Proposition~\ref{prop:first-energy}, the boundary terms satisfies
\begin{align*}
\left|
\int_{\Omega\cap\{|f|=\varepsilon\}}
H_j\varphi\nu_i\,\mathrm{d}S
\right|
&\le
C\frac{\chi'(\log\varepsilon)}{\varepsilon}
\int_{\Omega\cap\{|f|=\varepsilon\}}
|\partial f|\,\mathrm{d}S\\
&\le
C\frac{\chi'(\log\varepsilon)}{\varepsilon}\,\varepsilon
=
C\chi'(\log\varepsilon)
\longrightarrow0,
\end{align*}
because a nonnegative non-decreasing function that is integrable at
$-\infty$ must tend to zero there.  Letting $\varepsilon$ tend
to zero and using the $L^1$ convergence of both volume integrals, we have
\[
\int_UL_{ij}\varphi\,\mathrm{d} V
=-\int_UH_j\partial_i\varphi\,\mathrm{d} V.
\]
Thus $L_{ij}=\partial_iH_j=\partial_{ij}v$ in distributions.  All weak
second derivatives belong to $L^1_{\mathrm{loc}}$, and consequently
$v\in W^{2,1}_{\mathrm{loc}}(U)$.

Finally, since $\chi'\ge0$, we have
\[
\int_{-\infty}^{A}\chi'(s)\,\mathrm{d} s
=\chi(A)-\lim_{s\to-\infty}\chi(s).
\]
This proves the last equivalence in \eqref{eq:w21-criterion-intro}.
\end{proof}

\begin{remark}\label{rem:L1-not-CZ}
The proof uses the trace term controlled by $\chi''$ and the tangential
Hessian terms controlled by $\chi'$.  This is the analytic reason why
$\Delta v \in L^1_{\mathrm{loc}}$ does not by itself imply $v\in W^{2,1}_{\mathrm{loc}}$.
\end{remark}

\section{Applications}\label{sec:examples}
In this section, we give two applications of our main results.
\subsection{Counterexamples to Calder\'on-Zygmund theory and beyond}

Assume that $A<-e$ if needed. For $t>0$, set
\[
\log^{(1)}t=\log t,
\qquad
\log^{(j+1)}t=\log\bigl(\log^{(j)}t\bigr).
\]
The following functions are convex and non-decreasing on a sufficiently far left half-line:
\[
\chi_\alpha(s)=-(-s)^\alpha,
\qquad 0<\alpha<1,
\]
\[
\chi_{\log^{(m)}}(s)=-\log^{(m)}(-s),
\qquad m\geq 1,
\]
and
\[
\chi_{\mathrm{inv}}(s)=\frac{1}{\log(-s)}.
\]
Here $\chi_{\log}=\chi_{\log^{(1)}}$, and for the last function we assume $A<-e$.

For the power function,
\[
\chi_\alpha'(s)
   =\alpha(-s)^{\alpha-1},
\qquad
\chi_\alpha''(s)
   =\alpha(1-\alpha)(-s)^{\alpha-2}.
\]
For the logarithmic function,
\[
\chi_{\log}'(s)=\frac{1}{-s},
\qquad
\chi_{\log}''(s)=\frac{1}{(-s)^2}.
\]
More generally, for $m\geq1$,
\[
\chi_{\log^{(m)}}'(s)
 =
 \frac{1}
 {(-s)\displaystyle\prod_{j=1}^{m-1}\log^{(j)}(-s)},
\]
and
\[
\chi_{\log^{(m)}}''(s)
 =
 \frac{1}
 {(-s)^2\displaystyle\prod_{j=1}^{m-1}\log^{(j)}(-s)}
 \left(
  1+
  \sum_{k=1}^{m-1}
  \frac{1}
  {\displaystyle\prod_{j=1}^{k}\log^{(j)}(-s)}
 \right).
\]
As usual, an empty product is interpreted as $1$ and an empty sum as $0$, so these formulas also include the case $m=1$. Direct differentiation shows
\[
\chi_{\mathrm{inv}}'(s)
 =
 \frac{1}{(-s)\log^2(-s)}>0,
\qquad
\chi_{\mathrm{inv}}''(s)
 =
 \frac{\log(-s)+2}
 {(-s)^2\log^3(-s)}>0.
\]
Thus $\chi_{\mathrm{inv}}$ is convex and non-decreasing on $(-\infty,A]$. Moreover,
\[
\lim_{s\to-\infty}\chi_{\mathrm{inv}}(s)=0
\]
and
\[
\int_{-\infty}^{A}\chi_{\mathrm{inv}}'(s)\,ds
 =
 \left[\frac{1}{\log(-s)}\right]_{-\infty}^{A}
 =
 \frac{1}{\log(-A)}
 <\infty.
\]
In particular,
\[
\ell_{\chi_{\mathrm{inv}}}=0,
\qquad
\chi_{\mathrm{inv}}'\in L^1((-\infty,A)).
\]

The criteria yield the following summary of regularity properties:

\begin{table}[ht!]
\centering
\renewcommand{\arraystretch}{1.4}
\setlength{\tabcolsep}{12pt}
\begin{tabular}{l c c c}
\hline\hline
function & $G \in L^1_{\mathrm{loc}}$ & $v \in W^{1,2}_{\mathrm{loc}}$ & $v \in W^{2,1}_{\mathrm{loc}}$ \\
\hline
$-(-s)^{\alpha} \quad (0 < \alpha < 1)$ & yes & iff $\alpha < \frac{1}{2}$ & no \\
$-\log(-s)$ & yes & yes & no \\
$-\log^{(m)}(-s) \quad (m \ge 2)$ & yes & yes & no \\
$\frac{1}{\log(-s)} \quad (s \le A < -e)$ & yes & yes & yes \\
\hline\hline
\end{tabular}
\end{table}

\noindent

\begin{remark}
The counterexamples constructed in \cite{PSWW} are in fact covered by the general criteria obtained in our paper by choosing
$
\chi(s)=-\log(-s).
$
The two logarithmic constructions in \cite{PSWW}, corresponding respectively to failure and validity of the endpoint $W^{2,1}$-regularity, arise as particular choices of the function $\chi$ in our general framework.
\end{remark}

\begin{proof}[Proof of Corollary~\ref{cor:psww}]
    The proof of Corollary~\ref{cor:psww} follows easily from the table above.
\end{proof}

\begin{corollary}\label{cor:w21-implies-w12}
In the above setting, we have
\[
v\in W^{2,1}_{\mathrm{loc}}(U)\quad\Longrightarrow\quad
v\in W^{1,2}_{\mathrm{loc}}(U).
\]
The converse fails even for $f(z)=z$.
\end{corollary}

\begin{proof}
Suppose first that $v\in W_{\mathrm{loc}}^{2,1}(U)$. By Theorem $1.3$ (ii),
\[
\chi'\in L^1((-\infty,A)).
\]
Since $\chi$ is convex and non-decreasing, $\chi'$ is nonnegative and non-decreasing. Hence
\[
0\leq \chi'(s)\leq \chi'(A),
\qquad s\leq A,
\]
and therefore
\[
\int_{-\infty}^{A}|\chi'(s)|^2\,ds
\leq
\chi'(A)\int_{-\infty}^{A}\chi'(s)\,ds
<\infty.
\]
For $p=2$, the weighted condition in Theorem $1.2$ reduces precisely to
\[
\chi'\in L^2((-\infty,A)).
\]
It follows that $v\in W_{\mathrm{loc}}^{1,2}(U)$.

To see that the converse fails, take $f(z)=z$ and, after shrinking $U$, let
\[
\chi(s)=-\log(-s),
\qquad s\leq A<-e.
\]
Then by the table above 
\[
-\log(-\log|z|)\in W_{\mathrm{loc}}^{1,2}(U),
\quad
-\log(-\log|z|)\notin W_{\mathrm{loc}}^{2,1}(U).
\]
Thus the converse fails even for $f(z)=z$.
\end{proof}

\subsection{The local domain of the Monge--Amp\`ere operator}\label{sec:MA}

In this subsection $n\ge2$. For locally bounded psh functions the powers
$(dd^cw)^j$ are understood in the Bedford--Taylor sense, defined
recursively by
\[
(dd^cw)^0=1,\qquad
(dd^cw)^j=dd^c\bigl(w(dd^cw)^{j-1}\bigr).
\]
They are positive closed currents, and agree with the usual wedge
products for smooth functions; see \cite{BT82}. For complex differential forms we use
\[
d^c=\frac{i}{2}(\bar\partial-\partial),\qquad
dd^c=i\partial\bar\partial,\qquad
\omega_0=dd^c|z|^2.
\]
For a real-valued $C^1$ function $w$,
$dw\wedge d^cw=i\partial w\wedge\bar\partial w$ is positive, and
\begin{equation}\label{eq:app-energy-trace}
dw\wedge d^cw\wedge\omega_0^{n-1}
=\frac{|\nabla w|^2}{4n}\,\omega_0^n.
\end{equation}

\begin{definition}\label{def:app-D}
A negative psh function $w$ belongs to $\mathcal D(\Omega)$ if there is
a positive Radon measure $\mu$ on $\Omega$ such that, on every open subset
$V\Subset\Omega$, every sequence $w_j\in
C^\infty(V)\cap\operatorname{PSH}(V)$ decreasing to $w|_V$ satisfies
$(dd^cw_j)^n\rightharpoonup\mu|_V$.  We then write
$(dd^cw)^n=\mu$.
\end{definition}

This is the classical local domain with continuity under decreasing
approximation. The definition requires independence of the approximating
sequences.

\begin{definition}\label{def:app-E}
A bounded domain $\Omega$ is hyperconvex if it admits a continuous
negative psh exhaustion $\rho$, so that
$\{\rho<-c\}\Subset\Omega$ for every $c>0$. On such a domain, define
\[
\mathcal E_0(\Omega)=
\left\{w\in\operatorname{PSH}(\Omega)\cap L^\infty(\Omega):
w\le0,\ \lim_{z\to\partial\Omega}w(z)=0,
\ \int_\Omega(dd^cw)^n<\infty\right\}.
\]
The Cegrell class $\mathcal E(\Omega)$ consists of negative psh functions
$w$ such that every point has a neighborhood $V\Subset\Omega$ and a
sequence $w_j\in\mathcal E_0(\Omega)$ decreasing on $\Omega$ and
converging to $w$ on $V$, with
$\sup_j\int_\Omega(dd^cw_j)^n<\infty$.
\end{definition}

These definitions follow \cite{Cegrell04}. We use the following
characterization, with the normalization of $d^c$ fixed above.

\begin{theorem}[B{\l}ocki, \cite{Blocki06}, Theorems 1.1 and 2.4]
\label{thm:app-blocki}
Let $w<0$ be psh on an open set $\Omega\subset\mathbb C^n$, $n\ge2$.
Then $w\in\mathcal D(\Omega)$ if and only if each point has a neighborhood
$V$ admitting smooth negative psh functions $w_j\downarrow w$ such that,
for every $K\Subset V$ and $k=0,\ldots,n-2$,
\begin{equation}\label{eq:app-blocki-energy}
\sup_j\int_K(-w_j)^{n-k-2}
dw_j\wedge d^cw_j\wedge(dd^cw_j)^k
\wedge\omega_0^{n-k-1}<\infty.
\end{equation}
If $w\in\mathcal D(\Omega)$, these bounds hold for every local smooth
decreasing negative psh approximation. On a bounded hyperconvex domain,
$\mathcal D(\Omega)=\mathcal E(\Omega)$ for negative psh functions.
\end{theorem}

\begin{proof}[Proof of Theorem \ref{prop:app-D-criterion}]
For sufficiently small $\varepsilon>0$, set
\[
v_\varepsilon
=\chi\!\left(\frac12\log(|f|^2+\varepsilon^2)\right).
\]
On every relatively compact subdomain of $U$, these functions
are well defined, smooth, negative, and plurisubharmonic for
all sufficiently small $\varepsilon$. Moreover,
$v_\varepsilon\downarrow v$ as $\varepsilon \rightarrow 0$,
with convergence in $C^1$ on compact subsets of
$U\setminus Z(f)$.

Suppose first that $v\in\mathcal D(U)$.
Theorem~\ref{thm:app-blocki} with $k=0$, together with
\eqref{eq:app-energy-trace} and Fatou's lemma, implies
\[
\int_{K\setminus Z(f)}
(-v)^{n-2}|\nabla v|^2\,\mathrm{d}V<\infty
\]
for every $K\Subset U$.
Near a generic smooth point of a divisor component, choose a
holomorphic coordinate in which $f=w_1^m$.
Since a holomorphic coordinate change preserves this local
integrability, integration in the normal variable $r=|w_1|$
yields
\[
\int_0^{r_0}
[-\chi(m\log r)]^{n-2}
\chi'(m\log r)^2\,\frac{\mathrm{d}r}{r}<\infty.
\]
The substitution $s=m\log r$ proves the required condition.

Conversely, assume that the integral in
\eqref{eq:app-D-criterion} is finite.
Direct differentiation gives
\[
dv_\varepsilon\wedge d^cv_\varepsilon
=
\frac{|f|^2}
     {4(|f|^2+\varepsilon^2)^2}
\chi'\!\left(\frac12\log(|f|^2+\varepsilon^2)\right)^2
\,i\,df\wedge d\bar f
\]
and
\begin{align*}
dd^cv_\varepsilon
={}&
\left[
\frac{|f|^2}
     {4(|f|^2+\varepsilon^2)^2}
\chi''\!\left(\frac12\log(|f|^2+\varepsilon^2)\right)
\right.\\
&\left.\qquad+
\frac{\varepsilon^2}
     {2(|f|^2+\varepsilon^2)^2}
\chi'\!\left(\frac12\log(|f|^2+\varepsilon^2)\right)
\right]i\,df\wedge d\bar f.
\end{align*}
Both forms are scalar multiples of $i\,df\wedge d\bar f$,
whose square vanishes. Therefore, since $n\ge2$,
\begin{equation}\label{eq:app-rank-one}
dv_\varepsilon\wedge d^cv_\varepsilon
\wedge(dd^cv_\varepsilon)^k=0
\quad(k\ge1),
\qquad
(dd^cv_\varepsilon)^n=0.
\end{equation}

It remains to verify the energy bound for $k=0$.
Fix $K\Subset U$. By the coarea formula and
Proposition~\ref{prop:first-energy},
\begin{align*}
&\int_{K\cap\{|f|<r_0\}}
(-v_\varepsilon)^{n-2}
|\nabla v_\varepsilon|^2\,\mathrm{d}V\\
&\quad\le C\int_0^{r_0}
\left[-\chi\!\left(\frac12\log(r^2+\varepsilon^2)\right)\right]^{n-2}
\chi'\!\left(\frac12\log(r^2+\varepsilon^2)\right)^2
\frac{r^3\,\mathrm{d}r}{(r^2+\varepsilon^2)^2}\\
&\quad\le C\int_{-\infty}^{A}
(-\chi(s))^{n-2}\chi'(s)^2\,\mathrm{d}s.
\end{align*}
For the last inequality, use
$s=\frac12\log(r^2+\varepsilon^2)$, for which
\[
\frac{r^3}{(r^2+\varepsilon^2)^2}\,\mathrm{d}r
=
\frac{r^2}{r^2+\varepsilon^2}\,\mathrm{d}s
\le\mathrm{d}s.
\]
Here $C$ is independent of $\varepsilon$.
The contribution from $K\cap\{|f|\ge r_0\}$ is uniformly
bounded, since $v_\varepsilon\to v$ in $C^1$ there.
Together with \eqref{eq:app-energy-trace} and
\eqref{eq:app-rank-one}, this verifies all the bounds in
\eqref{eq:app-blocki-energy}.
From Theorem~\ref{thm:app-blocki} we therefore have
$v\in\mathcal D(U)$.

We now prove the two vanishing assertions.
Since $v_\varepsilon$ decreases to $v\in\mathcal D(U)$,
continuity under decreasing smooth approximation implies
\[
(dd^cv_\varepsilon)^n\rightharpoonup(dd^cv)^n.
\]
The left-hand side vanishes identically by
\eqref{eq:app-rank-one}, so $(dd^cv)^n=0$.

Furthermore, since $-\chi\ge2$,
\[
\int_{-\infty}^{A}\chi'(s)^2\,\mathrm{d}s
\le 2^{2-n}
\int_{-\infty}^{A}
(-\chi(s))^{n-2}\chi'(s)^2\,\mathrm{d}s<\infty.
\]
As $\chi'$ is nonnegative and non-decreasing, this implies
$\lim_{s\to-\infty}\chi'(s)=0$.
Consequently, for every $\delta>0$, there exists a constant
$C_\delta$ such that
\[
\chi(s)\ge\delta s-C_\delta,
\qquad s\le A.
\]
Fix $x\in Z(f)$. Since
$v\ge\delta\log|f|-C_\delta$ and
$\log|f(z)|\to-\infty$ as $z\to x$, we have
\[
\limsup_{\substack{z\to x\\z\notin Z(f)}}
\frac{v(z)}{\log|f(z)|}\le\delta.
\]
Apply the second comparison theorem
\cite[Chapter~III, Theorem~7.8]{Demailly}
with $q=1$, $T=1$, $u_1=\log|f|$, $v_1=v$,
and $\varphi(z)=\log|z-x|$.
The required condition on the unbounded loci holds because
both functions are locally bounded outside the analytic
hypersurface $Z(f)$. We obtain
\[
0\le\nu(v,x)
\le\delta\,\nu(\log|f|,x)
=\delta\,\operatorname{ord}_x(f),
\]
where the last equality follows from
\cite[Chapter~III, \S6.10]{Demailly}.
Letting $\delta\rightarrow0$, then we have $\nu(v,x)=0$.
At points outside $Z(f)$, the same conclusion holds
because $v$ is smooth. The proof is completed.
\end{proof}

\end{document}